\documentclass[reqno]{amsart}
\usepackage{amsthm,amsmath,amsmath,amssymb,amsbsy,graphics,hyperref}
\newcommand{\comment}[1]{{}}

\theoremstyle{plain}
\newtheorem{theorem}{Theorem}
\numberwithin{theorem}{section}
\newtheorem{lemma}[theorem]{Lemma}
\newtheorem{proposition}[theorem]{Proposition}
\newtheorem{corollary}[theorem]{Corollary}
\newtheorem{conjecture}[theorem]{Conjecture}

\theoremstyle{definition}
\newtheorem{definition}[theorem]{Definition}

\newtheorem{example}[theorem]{Example}

\theoremstyle{remark}

\newtheorem{problem}[theorem]{Problem}

\newcommand{\bd}{\partial^{\phantom{*}}} 
\newcommand{\bdo}{\partial} 
\newcommand{\cbd}{\partial^*}
\newcommand{\wbd}{\hat\partial} 
\newcommand{\wcbd}{\hat\partial^*}

\newcommand{\stot}{s^{\rm tot}} 
\newcommand{\sud}{s^{\rm ud}} 
\newcommand{\sdu}{s^{\rm du}} 
\newcommand{\sany}{s^\bullet} 
\newcommand{\Ltot}{L^{\rm tot}} 
\newcommand{\Lud}{L^{\rm ud}}
\newcommand{\Ldu}{L^{\rm du}}
\newcommand{\Lany}{L^\bullet}
\newcommand{\Etot}{E^{\rm tot}} 
\newcommand{\Eud}{E^{\rm ud}}
\newcommand{\Edu}{E^{\rm du}}
\newcommand{\Eany}{E^\bullet}

\newcommand{\Swtot}{\hat{s}^{\rm tot}} 
\newcommand{\Swud}{\hat{s}^{\rm ud}}

\newcommand{\Swany}{\hat{s}^\bullet}
\newcommand{\Lwtot}{\hat{L}^{\rm tot}} 
\newcommand{\Lwud}{\hat{L}^{\rm ud}}
\newcommand{\Lwdu}{\hat{L}^{\rm du}}
\newcommand{\Lwany}{\hat{L}^\bullet}
\newcommand{\Ewtot}{\hat{E}^{\rm tot}} 
\newcommand{\Ewud}{\hat{E}^{\rm ud}}

\newcommand{\Ewany}{\hat{E}^\bullet}

\newcommand{\Stot}{\mathbf{s}^{\rm tot}} 

\DeclareMathOperator{\dir}{dir}
\DeclareMathOperator{\id}{id}
\DeclareMathOperator{\im}{im}
\DeclareMathOperator{\rank}{rank}

\DeclareMathOperator{\Spec}{Spec}
\DeclareMathOperator{\del}{del}
\DeclareMathOperator{\link}{link}

\DeclareMathOperator{\bnd}{bd}

\newcommand{\CST}{{\mathcal T}}
\newcommand{\HH}{\tilde H} 
\newcommand{\Prism}{{\bf P}}
\newcommand{\PO}{{\bf 0}}
\newcommand{\PI}{{\bf 1}}

\newcommand{\Del}[2]{\del_{#1}{#2}} 
\newcommand{\Link}[2]{\link_{#1}{#2}}

\newcommand{\Nn}{\mathbb{N}}
\newcommand{\Qq}{\mathbb{Q}}
\newcommand{\Rr}{\mathbb{R}}
\newcommand{\Zz}{\mathbb{Z}}

\newcommand{\q}{r} 

\newcommand{\betti}{\tilde\beta} 

\newcommand{\zeq}{\overset{\circ}{=}}
\newcommand{\st}{\colon}
\newcommand{\qandq}{\quad\text{and}\quad}
\newcommand{\0}{\emptyset}
\newcommand{\X}{\circledast} 

\newcommand{\sign}{\varepsilon}

\newcommand{\x}{\times}

\newcommand{\sm}{\backslash}

\begin{document}

\date{May 5, 2010}

\title{Cellular Spanning Trees and Laplacians of Cubical Complexes}

\author{Art M.\ Duval}
\address{Department of Mathematical Sciences,
University of Texas at El Paso,
El Paso, TX 79968-0514}

\author{Caroline J.\ Klivans}
\address{Departments of Mathematics and Computer Science,
The University of Chicago,
Chicago, IL 60637}

\author{Jeremy L.\ Martin}
\address{Department of Mathematics,
University of Kansas,
Lawrence, KS 66047}

\thanks{Third author partially supported by an NSA Young Investigators Grant.}

\dedicatory{``Is there a $q$-analogue of that?'' --- Dennis Stanton.}

\begin{abstract}
We prove a Matrix-Tree Theorem enumerating the spanning trees of
a cell complex in terms of the eigenvalues of its cellular Laplacian
operators, generalizing a previous result for simplicial complexes.
As an application, we obtain explicit formulas for spanning
tree enumerators and Laplacian eigenvalues of cubes; the latter
are integers.  We prove a weighted version of the eigenvalue
formula, providing evidence for a conjecture on weighted
enumeration of cubical spanning trees.
We introduce a cubical analogue of shiftedness, and obtain a recursive
formula for the Laplacian eigenvalues of shifted cubical complexes,
in particular, these eigenvalues are also integers.
Finally, we recover Adin's enumeration of spanning trees of a complete
colorful simplicial complex from the cellular Matrix-Tree Theorem together
with a result of Kook, Reiner and Stanton.
\end{abstract}

\maketitle

\section{Introduction} \label{intro}

\subsection{Cellular spanning trees}
In~\cite{DKM}, the authors initiated the study of \emph{simplicial spanning
trees}: subcomplexes of a simplicial complex that behave much like the
spanning trees of a graph.  The central result of \cite{DKM} is a
generalization of the Matrix-Tree Theorem, enumerating
simplicial spanning trees in terms of eigenvalues of
combinatorial Laplacians.  In this paper, we extend our field of inquiry
to the setting of arbitrary cell complexes and their Laplacians.

Let $X$ be a $d$-dimensional cell complex; we write $X_i$ for the set of $i$-cells of $X$.
A \emph{cellular spanning tree}, or just a \emph{spanning tree}, of
$X$ is a $d$-dimensional subcomplex $Y\subseteq X$ with the same $(d-1)$-skeleton,
whose (reduced) homology satisfies the two conditions $\HH_d(Y;\Zz) = 0$
and $|\HH_{d-1}(Y;\Zz)| < \infty$.  These two conditions
imply that $|Y_d| = |X_d| - \betti_d(X) + \betti_{d-1}(X)$,
where $\betti_d$ and $\betti_{d-1}$ denote (reduced) Betti numbers.
In fact, any two of these three conditions together imply the third.
In the case $d=1$, a cellular spanning tree is just a spanning tree of $X$ in the
usual graph-theoretic sense; the first two conditions above say
respectively that $Y$ is acyclic and connected, and the third condition
says that $Y$ has one fewer edge than the number of vertices in $X$.

Let $C_i$ denote the $i^{th}$ cellular chain group of $X$ with coefficients in $\Zz$, and let $\bd_i\colon C_i \to C_{i-1}$ and $\cbd_i\colon C_{i-1}
\to C_i$ be the cellular boundary and coboundary maps
(where we have identified cochains with chains via the natural inner product).
The \emph{$i^{th}$ up-down, down-up and total combinatorial Laplacians}
are respectively
  $$\Lud_i=\bd_{i+1}\cbd_{i+1}, \qquad
    \Ldu_i=\cbd_i\bd_i,\qquad
    \Ltot_i = \Lud_i+\Ldu_i,$$
which may be viewed either as endomorphisms on $C_i$, or as square symmetric
matrices, as convenient.  We are interested in the \emph{spectra} of these
Laplacians, that is, their multisets of eigenvalues, which we denote by
$\sud_i(X)$,  $\sdu_i(X)$, and $\stot_i(X)$ respectively.
Combinatorial Laplacians seem to have first appeared in the
work of Eckmann~\cite{Eck} on finite-dimensional Hodge theory, in which
the $i^{th}$ homology group of a chain complex
is identified with $\ker(L_i)$ via the direct sum decomposition
$C_i=\im\bd_{i+1} \oplus \ker L_i \oplus \im\cbd_i$.
As the name suggests, the combinatorial Laplacian is a discrete version
of the Laplacian on differential forms for a Riemannian manifold;
Dodziuk and Patodi~\cite{DP} proved that for suitably nice triangulations, the
eigenvalues of the combinatorial and analytic Laplacians converge to each other.

For $0\leq k\leq d$, let $\CST_k(X)$ denote the set of all spanning $k$-trees of $X$
(that is, the spanning trees of the $k$-skeleton of~$X$).
Let
\begin{align*}
  \tau_k(X) &= \sum_{Y\in\CST_k(X)} |\HH_{k-1}(Y;\Zz)|^2,\\
  \pi_k(X) &= \prod_{0\neq\lambda\in\sud_{k-1}(X)} \lambda \qquad\text{for } k\geq 1.
\end{align*}
We set $\pi_0=|X_0|$, the number of vertices of~$X$,
for reasons that will become clear later.
Note that $\tau_1(X)$ is just the number of spanning trees of the 1-skeleton
of $X$.  Bolker \cite{Bolker} was the first to observe that enumeration of
higher-dimensional trees requires some consideration of torsion; the
specific summand $|\HH_{k-1}(Y;\Zz)|^2$, first noticed by Kalai \cite{Kalai},
arises from an application of the Binet--Cauchy theorem.
The precise relationship between the families of invariants $\{\pi_k(X)\}$
and $\{\tau_k(X)\}$ is as follows.

\medskip
\noindent\textbf{Theorem~\ref{thm:CMTT} [Cellular Matrix-Tree Theorem]}.\ 
\textit{
Let $d\geq 1$, and let $X^d$ be a cell complex such that $H_i(X;\Qq)=0$ for all $i<d$.  Fix a spanning $(d-1)$-tree $Y\subset X$, let $U=Y_{d-1}$, and let $L_U$ be the matrix obtained from $\Lud_{d-1}(X)$ by deleting the rows and columns corresponding to $U$.  Then:}
\begin{align*}  
\pi_d(X) &= \frac{\tau_d(X) \tau_{d-1}(X)}{|\HH_{d-2}(X;\Zz)|^2} \qquad{\rm and}\\
\tau_d(X) &= \frac{|\HH_{d-2}(X;\Zz)|^2}{|\HH_{d-2}(Y;\Zz)|^2} \det L_U.
\end{align*}
\medskip

These formulas are identical to those of the Simplicial Matrix-Tree Theorem
\cite[Theorem~1.3]{DKM}, except that they apply to the wider class
of cell complexes.  Note that the classical Matrix-Tree Theorem for graphs is
just the special case $d=1$.  Our proof, in Section~\ref{cellular-section},
is similar to that in \cite{DKM}, and ultimately stems from the work of
Kalai~\cite{Kalai}.  It is not hard to generalize the result to a weighted
version, Theorem~\ref{thm:WCMTT}.

Of particular interest to us is the problem of which complexes are
\emph{Laplacian integral}, that is, all of whose Laplacians have
integer eigenvalues.  This question seems to have been first raised
by Friedman~\cite{Friedman}.
While no general characterization of Laplacian integrality is known,
many classes of simplicial complexes are known to have combinatorially
meaningful Laplacian integer spectra.  These
include matroid complexes~\cite{KRS}, 
shifted complexes~\cite{DR}, matching complexes~\cite{Dong-Wachs},
and chessboard complexes~\cite{Friedman-Hanlon}.

In Section~\ref{cellular-section} of the paper, we develop this machinery more carefully
and prove the Cellular Matrix-Tree Theorem and its weighted analogue.

\subsection{Cubical complexes}
In Section~\ref{sec:cubical}, we study Laplacian integrality for cell
complexes that are cubical rather than simplicial.  Specifically, the
$n$-dimensional unit cube $Q_n\subset\Rr^n$ can be regarded as a cell complex whose faces
are indexed by the ordered $n$-tuples $f=(f_1,\dots,f_n)$,
where $f_i\in\{0,1,\X\}$, with $f\subseteq g$ iff
$f_i=g_i$ whenever $g_i\neq\X$; here
$\dim f=\#\{i\st f_i=\X\}$.  
Intuitively, we think of $\X$ as the open interval $(0,1)$.
A \emph{proper cubical complex} is
then an order ideal in this face poset.  (This is a more restricted class
than the more usual definition of ``cubical complex'', meaning any cell
complex in which the closure of each cell is a cube.)

For a proper cubical complex $X\subseteq Q_n$, the \emph{prism} over $X$
is the cell complex $\Prism X:=X\x\{0,1,\X\}\subseteq Q_{n+1}$.
The cellular Laplacian spectra of $X$ determine those of
$\Prism X$; the precise formulas are given in Theorem~\ref{prism-spectrum}.
Since the cube $Q_n$ itself is an iterated prism,
it follows from these results that $Q_n$ and all its skeletons are Laplacian
integral, with spectra given by the formulas
\begin{align}
 \label{Qskel:1}
\sum_{\lambda\in \stot_k(Q_n)} \!\!\! \q^\lambda &= \binom{n}{k} \q^{2k} (1+\q^2)^{n-k},\\
 \label{Qskel:2}
\sum_{\lambda\in \sud_k(Q_n)} \!\!\! \q^\lambda &\zeq
    \sum_{i=k+1}^n \binom{n}{i} \binom{i-1}{k} \q^{2i}
\end{align}
(Theorem~\ref{cube-eigenvalue-theorem}),
where $\zeq$ means ``equal up to constant coefficient'' (i.e.,
up to the multiplicity of zero as an eigenvalue).
Applying the Cellular Matrix-Tree Theorem to formula~\eqref{Qskel:2} gives
  $$\tau_k(Q_n) = \prod_{j=k+1}^n (2j)^{\binom{n}{j} \binom{j-2}{k-1}}$$
(Corollary~\ref{thm:general-count}).  This generalizes the well-known count
$\prod_{j=2}^n (2j)^{\binom{n}{j}}$
of spanning trees of the $n$-cube graph (see, e.g., \cite[Example~5.6.10]{EC2}).

\subsection{Weighted spanning tree enumerators}
In Section~\ref{sec:weighted}, we discuss weighted spanning tree enumerators
of skeletons of cubes.

To each face $f=(f_1,\dots,f_n)\in Q_n$, associate the monomial
  $$\xi_f = \prod_{i:f_i=\X} q_i \prod_{i:f_i=0} x_i \prod_{i:f_i=1} y_i,$$
in commuting indeterminates.  For a cubical subcomplex $X\subseteq Q_n$,
we consider the weighted spanning tree enumerator
  $$\hat\tau_k(X) = \sum_{Y\in\CST_k(X)} \prod_{f\in Y_k} \xi_f.$$
In principle, one can calculate $\hat\tau_k(X)$ by replacing the Laplacians of~$X$
with their weighted versions and applying Theorem~\ref{thm:WCMTT}.  In practice,
the difficulty is that the eigenvalues of the weighted Laplacians of cubical
complexes --- even for such basic cases as skeletons of cubes --- are in general
 not polynomial.  Furthermore, the corresponding weighted cellular boundary maps
$\partial$ do not satisfy $\partial^2=0$, and so do not give the structure of an
algebraic chain complex.

In order to overcome this difficulty we replace the natural ``combinatorial''
weighting $\xi_f$ with a closely related ``algebraic'' weighting by Laurent monomials,
whose boundary maps do satisfy $\partial^2=0$.
The algebraically weighted Laplacian spectra of $Q_n$ is given explicitly as
follows (Theorem~\ref{cube-weighted-eigenvalue-theorem}):
$$
\sum_{\lambda\in \Swtot_k(Q_n)} \!\!\! \q^\lambda 
 = \sum_{j=i}^n \binom{j}{i} e_j(\q^{u_1},\dots,\q^{u_n}),$$
where $u_i = q_i^2/x_i^2 + q_i^2/y_i^2$
and $e_j$ denotes the $j^{th}$ elementary symmetric function.

In the case of shifted simplicial complexes, a weighted version of the Simplicial Matrix-Tree Theorem
can be used to translate the
knowledge of the algebraically weighted Laplacian spectrum into an exact
combinatorial formula for the weighted tree enumerators $\tau_k(X)$
(see \cite{DKM}).  We had hoped that this approach would succeed in the case
of skeletons of cubes, but the exact form of the relationship
between the two weightings has so far eluded us.  On the other hand, there is
strong evidence that the combinatorial
spanning tree enumerator $\tau_k(Q_n)$ is given by the following formula (Conjecture~\ref{conj:new}):

$$\hat\tau_k(Q_n) =
  (q_1\cdots q_n)^{\displaystyle \sum_{i=k-1}^{n-1}\binom{n-1}{i}\binom{i-1}{k-2}}
  \prod_{\substack{A\subseteq[n]\\|A|\geq k+1}} \left[
  \sum_{i\in A} \left(q_i(x_i+y_i)\prod_{j\in A\sm i} x_j y_j \right) \right]^{\displaystyle{\binom{|A|-2}{k-1}}}.$$

\subsection{Shifted cubical complexes}
In Section~\ref{sec:shifted}, we introduce an analogue of shiftedness
for cubical complexes.  

Shifted complexes are a well-known and important family of
simplicial complexes; see, e.g., \cite{BK,Kalai_shifting}
as general references.  They may be defined 
as certain iterated near-cones, or as order ideals with respect
to a natural partial ordering of integer sets.
The Laplacian spectrum of a shifted simplicial complex $\Delta$ can be
expressed recursively in terms of the spectra of its link and deletion
\cite[Lemma~5.3]{DR}, and explicitly as the
transpose of its vertex-facet degree sequence \cite[Theorem~1.1]{DR}.
In particular, the eigenvalues are integers.

Our definition of a cubical analogue of shiftedness seeks to extend all these
properties of shifted simplicial complexes to the cubical setting.  First,
shifted subcomplexes of $Q_n$ are
in bijection with the shifted simplicial complexes on $n$ vertices,
via the ``mirroring'' operation as described in Babson, Billera, and
Chan \cite{BBC}; roughly speaking, vertices of a simplicial complex
correspond to directions of a cubical complex.  Like their simplicial
cousins, shifted cubical complexes are a particular type of iterated
near-prism.  Moreover, their eigenvalues may be described recursively
in terms of the cubical link and deletion.
An immediate corollary of this recurrence is that
shifted cubical complexes are Laplacian
integral.
We present a number of natural questions for further research on shifted cubical complexes.

\subsection{Duality}
In Section~\ref{sec:duality}, we study the connections between the Laplacian
spectra of a cell complex and its dual complex.

Two cell complexes $X,Y$ are \emph{dual} if there is an
inclusion-reversing bijection
$f\mapsto f^*$ from the cells of~$X$ to the cells of $Y$,
with $\dim f+\dim f^*=d$ for all $f$, such that
the boundary maps of $X$ equal the coboundary maps of $Y$.  In this case,
we have an immediate duality on the Laplacian spectra:
$\stot_i(X)=\stot_{d-i}(Y)$ for all~$i$.
Moreover, for subsets $T\subseteq X_i$ and $U=\{f^* \st f\in X_i\sm T\}$,
the subcomplex $X_T=T\cup X_{(i-1)}$ forms a spanning $i$-tree of~$X$
if and only if
the subcomplex $Y_U=U\cup Y_{(d-i-1)}$ forms a spanning $(d-i)$-tree of~$Y$
(Proposition~\ref{dual-complexes}).\footnote{As pointed out by a referee,
$Y_U=\{f^*\st f\not\in X_T\}$, a relationship reminiscent of Alexander
duality for simplicial complexes.}
From a matroid point of view,
this is essentially the statement that
the matroid represented by the columns of
$\bd_{X,i}$ is dual to the matroid represented by the columns of $\bd_{Y,d-i}$;
the bases of these matroids are cellular spanning trees.
(If $X$ is a simplicial complex, this matroid is known as a \emph{simplicial
matroid}; see, e.g., \cite{CL}.)

The cellular dual of the $n$-cube is the $n$-dimensional cross-polytope.
This is an instance of a \emph{complete colorful complex},
whose spanning tree enumerators were calculated by Adin~\cite{Adin}.
In addition, complete colorful complexes are matroid complexes (corresponding
to products of rank-1 matroids). Kook, Reiner and Stanton determined
the Laplacian spectra of matroid complexes in \cite{KRS}.  We explore
the connection between all these results.  In particular,
we explain how to derive
Adin's formula from the Kook-Reiner-Stanton formula together
with the Cellular Matrix-Tree Theorem, and give an alternate calculation
of the (unweighted) spectra of cubes (Theorem~\ref{cube-eigenvalue-theorem})
using duality together with Adin's formula.

We thank Gil Kalai for suggesting cubical complexes as fertile ground
for exploration; Eran Nevo for solving Problem \ref{homotopy_problem};
and two anonymous referees for their careful reading and many
valuable suggestions.

\section{Cellular spanning trees}
\label{cellular-section}

\subsection{Cell complexes}

See Hatcher \cite{Hatcher} for definitions and basic facts about cell complexes.
We write $X_i$ for the set of all $i$-dimensional cells of $X$, and
$X_{(i)}$ instead of Hatcher's $X^i$ for the $i$-skeleton $X_0\cup X_1\cup\cdots\cup X_i$.
The notation $X^d$ indicates that $X$ has dimension $d$.
We borrow some standard terminology from the theory of simplicial complexes:
a cell of $X$ not contained in the closure of any other cell is called a \emph{facet}
of~$X$, and we say that $X$ is \emph{pure} if all its facets have the same dimension.

Let $R$ be a ring (if unspecified, assumed to be $\Zz$), and let $C_i(X)$ denote the $i^{th}$ cellular chain group of $X$, i.e., the free $R$-module with basis $\{[F] \st F\in X_i\}$.  We then have cellular
boundary and coboundary maps
  \begin{align*}
  \bd_{X,i}  &\colon C_i(X) \to C_{i-1}(X),\\
  \cbd_{X,i} &\colon C_{i-1}(X) \to C_i(X),
  \end{align*}
where we have identified cochains with chains via the natural
inner product.  We will abbreviate the subscripts in
the notation for boundaries and coboundaries whenever no ambiguity can arise.  

We will often regard $\bd_i$ (resp.\ $\cbd_i$)
as a matrix whose columns and rows (resp.\ rows and columns)
are indexed by $X_i$ and $X_{i-1}$ respectively.

The \emph{$i^{th}$ (reduced) homology group} of $X$ is
$\HH_i(X)=\ker(\bd_i)/\im(\bd_{i+1})$, and the \emph{$i^{th}$
(reduced) Betti number} $\betti_i(X)$ is the rank of the
largest free $R$-module summand of $\HH_i(X)$.

\subsection{Laplacians}

For $0\leq i\leq\dim X$,
define linear operators $\Lud_{X,i}$, $\Ldu_{X,i}$, $\Ltot_{X,i}$ on the
vector space $C_i(X)$ by
\begin{align*}
& \Lud_{X,i}  = \bd_{i+1}\cbd_{i+1}     && (\text{the \emph{up-down Laplacian}}),\\
& \Ldu_{X,i}  = \cbd_i\bd_i             && (\text{the \emph{down-up Laplacian}}),\\
& \Ltot_{X,i} = \Lud_{X,i} + \Ldu_{X,i} && (\text{the \emph{total Laplacian}}).
\end{align*}

Let $\sud_i(X)$,  $\sdu_i(X)$ and $\stot_i(X)$ denote the \emph{spectra} of
the respective Laplacians, that is, the multisets of their eigenvalues.
Each of these multisets has cardinality $|X_i|$ (counting multiplicities),
because the Laplacians can be represented by symmetric square matrices of that size. 
Instead of working with multisets, it is often convenient to record the eigenvalues
by the generating functions
\begin{align*}
\Eany_i(X;\q) &= \sum_{\lambda\in \sany_i(X)} \!\!\! \q^\lambda,\\
\Eany(X;\q,t) &= \sum_{i=0}^{\dim X}\sum_{\lambda\in \sany_i(X)} \!\!\! t^i \q^\lambda
\end{align*}
where $\bullet\in\{\text{ud},\text{du},\text{tot}\}$. 

The various Laplacian spectra are related by the identities
  \begin{subequations}
  \begin{align}
  \Eud_i(X;\q)  &\zeq \Edu_{i+1}(X;\q),\\
  \Etot_i(X;\q) &\zeq \Eud_i(X;\q) + \Edu_i(X;\q) = \Eud_i(X;\q) + \Eud_{i-1}(X;\q), \label{tot-from-uddu}\\
  \Eud_i(X;\q)  &\zeq \Etot_i(X;\q) - \Eud_{i-1}(X;\q), \label{ud-recurrence}
  \end{align}
  \end{subequations}
\cite[eqn.~(3.6)]{DR}, where $f\zeq g$ means that the two polynomials
$f,g$ are equal up to their constant coefficients.
It follows that
  \begin{equation} \label{from-total-to-ud}
  \Edu_{i+1}(X;\q) \zeq \Eud_i(X;\q) \zeq \sum_{j=0}^i (-1)^{i-j} \Etot_j(X;\q).
  \end{equation}
Therefore, each of the three families of generating funtions
  $$\{\Etot_i(X;\q) \st 0\leq i\leq\dim X\},\quad
    \{\Eud_i(X;\q) \st 0\leq i\leq\dim X\},\quad
    \{\Edu_i(X;\q) \st 0\leq i\leq\dim X\}$$
determines the other two.  (The constant coefficients, which represent
the multiplicity of the zero eigenvalue, can always be found by
observing that $\Eud_i(X,1)=\Edu_i(X,1)=\Etot_i(X,1)=|X_i|$.)

The bookkeeping differs slightly from simplicial complexes:
each cell in a cell complex has dimension $\geq0$, whereas every
simplicial complex includes the $(-1)$-dimensional face $\0$.
Therefore, for instance, $\Edu_0(X,q)\zeq0$ for all $X$.


\subsection{Product and sum formulas for the total Laplacian}
\label{product-sum}

For this section, it is convenient to regard the boundary map
of a cell complex $X$ as a linear endomorphism $\bd_X\colon C(X)\to C(X)$,
where $C(X)=\bigoplus_i C_i(X)$.  That is, $\bd_X=\sum_i\bd_{X,i}$;
equivalently, for each~$i$, $\bd_{X,i}$ is the restriction of $\bd_X$ to the group
$C_i(X)$ of cellular $i$-chains.

If $X$ and $Y$ are disjoint cell complexes, then
  $$\Etot(X\cup Y;\q,t)=\Etot(X;\q,t)+\Etot(Y;\q,t),$$
because the $k^{th}$
total boundary map (resp.,\ Laplacian) of $X\cup Y$ is just the
direct sum of the corresponding boundary maps (resp.,\ Laplacians)
of $X$ and $Y$.

Let $Z=X\x Y$.  Then
the set of $k$-cells of $Z$ is
  $$Z_k = \bigcup_{i=0}^k X_i \x Y_{k-i},$$
and on the level of chain groups we have
  $$C_k(Z) = \bigoplus_{i=0}^k C_i(X) \otimes C_{k-i}(Y).$$
The boundary map $\bd_Z$ is defined as follows.
Let $f\in X_i$ and $g\in Y_{k-i}$.  Then $(f,g)\in Z_k$, and the boundary
map acts on the corresponding cellular chain $[f]\otimes[g]\in C_k(Z)$ by
$\dim f=i$ and $\dim g=k-i$, 
  $$\bd_Z([f]\otimes[g]) = \bd_X[f]\otimes[g] + (-1)^{\dim f}[f]\otimes\bd_Y[g].$$
More simply, we may write
  $$\bd_Z  = \sum_{i=0}^k \bd_X\otimes\id + (-1)^i\id\otimes\bd_Y,\qquad
    \cbd_Z = \sum_{i=0}^k \cbd_X\otimes\id + (-1)^i\id\otimes\cbd_Y.$$
From this we calculate
  \begin{align*}
  \Lud_Z ~=~ \bd_Z\cbd_Z
  &= (-1)^{i-1} \bd_X\otimes\cbd_Y + \Lud_X\otimes\id + \id\otimes\Lud_Y + (-1)^i \cbd_X\otimes\bd_Y,\\
  \Ldu_Z ~=~ \cbd_Z\bd_Z
  &= (-1)^i \bd_X\otimes\cbd_Y + \Ldu_X\otimes\id + \id\otimes\Ldu_Y + (-1)^{i-1} \cbd_X\otimes\bd_Y,\\
  \Ltot_Z &= \Ltot_X\otimes\id + \id\otimes\Ltot_Y.
  \end{align*}
That is, the matrix for $\Ltot_Z$ is block-diagonal, with blocks
  $$\Ltot_{X,0}\otimes\id + \id\otimes\Ltot_{Y,k}, \qquad
    \Ltot_{X,1}\otimes\id + \id\otimes\Ltot_{Y,k-1}, \qquad\dots,\quad
    \Ltot_{X,k}\otimes\id + \id\otimes\Ltot_{Y,0}.$$
Therefore,
  $$\stot_k(Z) = \{\lambda+\mu \st \lambda\in\stot_i(X),\ \mu\in\stot_j(Y),\ i+j=k\}$$
as multisets.  From this we obtain the following product formula.

\begin{theorem}
\label{product-laplacian}
Let $X,Y$ be cell complexes.  Then
  $$\Etot(X\x Y;\q,t) = \Etot(X;\q,t)\Etot(Y;\q,t).$$
\end{theorem}


\subsection{Spanning trees}

\begin{definition} \label{def_sst}
Let $X^d$ be a cell complex, and let $k\le d$.
A \emph{cellular spanning $k$-tree} (for short, CST or $k$-CST) of
$X$ is a $k$-dimensional subcomplex $Y\subseteq X$ such that
$Y_{(k-1)} = X_{(k-1)}$ and
  \begin{subequations}
  \begin{align}
  \label{acyc-condn}  & \HH_k(Y) = 0,\\
  \label{conn-condn}  & |\HH_{k-1}(Y)| < \infty, \quad\text{and}\\
  \label{count-condn} & |Y_k| = |X_k| - \betti_k(X) + \betti_{k-1}(X).
  \end{align}
  \end{subequations}

We write $\CST_k(X)$ for the set of all $k$-CST's of $X$,
sometimes omitting the subscript if $k=d$.
Note that $\CST_k(X)=\CST_k(X_{(j)})$ for all $j\geq k$.  
\end{definition}

A 0-dimensional CST is just a vertex of $X$.
If $X$ is a connected 1-dimensional cell complex---that is, a
connected graph---then
the definition of 1-CST coincides with the usual definition of spanning tree:
a spanning subgraph of $X$ that is acyclic, connected, and has one
fewer edge than the number of vertices in~$X$.

\begin{proposition} \label{two-out-of-three}
Let $Y\subseteq X$ be a $k$-dimensional subcomplex with $Y_{(k-1)}=X_{(k-1)}$.
Then any two of the conditions \eqref{acyc-condn}, \eqref{conn-condn},
\eqref{count-condn} together imply the third.
\end{proposition}

The proof is identical to that of \cite[Prop.~3.5]{DKM}.

\begin{definition}
A cell complex $X$ is \emph{acyclic in positive codimension}, or APC for
short, if $\betti_j(X)=0$ for all $j<\dim X$.
\end{definition}

As in the setting of simplicial complexes, APC-ness is the ``right'' analogue
of connectedness for graphs, in the following sense.

\begin{proposition} \label{metaconn}
A cell complex has a cellular spanning tree if and only if it is APC.
\end{proposition}

The proof is identical to that of \cite[Prop.~3.7]{DKM}.

\subsection{The Cellular Matrix-Tree Theorem}
\label{CMTT-section}

Throughout this section, let $X^d$ be an APC cell complex with $d\geq 1$.
Define
\begin{align*}
  \tau_k &=\tau_k(X) = \sum_{Y\in\CST_k(X)} |\HH_{k-1}(Y)|^2& \text{for } 0\leq k\leq d,\\
  \pi_k &=\pi_k(X) = \prod_{0\neq\lambda\in\sud_{k-1}(X)} \lambda& \text{for } 1\leq k\leq d.
\end{align*}
Observe that $\tau_0=|X_0|$, because a 0-dimensional CST is just a vertex.
In addition, we define
$$\pi_0=\pi_0(X)=|X_0|,\qquad \pi_{-1}=\pi_{-1}(X)=1.$$
While it might seem more consistent
to define $\pi_0=1$ (because $\sud_{-1}(X)=\0$), several subsequent enumeration
formulas (such as Corollary~\ref{thm:general-count} and
Theorem~\ref{count-colorful-trees}) can be stated much more conveniently with this convention
for~$\pi_0$.  (By
way of motivation, in the case of a \emph{simplicial} complex on $n$ vertices (see \cite{DKM}),
it is in fact true that $\pi_0=n$, because of the
presence of a $(-1)$-dimensional face; the Laplacian
$\Lud_{-1}$ is is not a void matrix, but rather the $1\x 1$ matrix with
single entry~$n$.)

Abbreviate $\betti_i=\betti_i(X)$ 
 and $\bdo=\bd_{X,d}$.
Let $T$ be a set of $d$-cells of $X$ of cardinality $|X_d|-\betti_d+\betti_{d-1}=|X_d|-\betti_d$,
and let $S$ be a set of $(d-1)$-cells such that $|S|=|T|$.  Define
$$
X_T = T \cup X_{(d-1)},\qquad
\bar{S} = X_{(d-1)}\setminus S,\qquad
X_{\bar{S}} = \bar{S} \cup X_{(d-2)},
$$
and let $\bd_{S,T}$ be the square submatrix of $\bdo$ with rows indexed by $S$ and columns
indexed by $T$.

\begin{proposition}  \label{nonsingular-criterion}
The matrix $\bd_{S,T}$ is nonsingular if and only if $X_T\in\CST_d(X)$ and $X_{\bar{S}}\in
\CST_{d-1}(X)$.
\end{proposition} 

\begin{proof}
We may regard $\bd_{S,T}$ as the top boundary map of the $d$-dimensional
relative complex $\Gamma =(X_T,X_{\bar{S}})$.  So $\bd_{S,T}$ is nonsingular
if and only if $\HH_d(\Gamma)=0$.  Consider the long exact sequence
  \begin{equation} \label{long-exact}
  0 \to \HH_d(X_{\bar{S}}) \to \HH_d(X_T) \to \HH_d(\Gamma)
    \to \HH_{d-1}(X_{\bar{S}}) \to \HH_{d-1}(X_T) \to \HH_{d-1}(\Gamma)
    \to \cdots .
  \end{equation}
If $\HH_d(\Gamma)\neq 0$, then $\HH_d(X_T)$ and $\HH_{d-1}(X_{\bar{S}})$
cannot both be zero.  This proves the ``only if'' direction.
 
If $\HH_d(\Gamma)=0$, then $\HH_d(X_{\bar{S}})=0$ (since $\dim X_{\bar{S}} = d-1$), so
\eqref{long-exact} implies $\HH_d(X_T)=0$.  Therefore $X_T$ is a $d$-tree,
because it has the correct number of facets.  Hence $\HH_{d-1}(X_T)$ is finite.
Then \eqref{long-exact} implies that $\HH_{d-1}(X_{\bar{S}})$ is finite.  In fact, it is zero
because the top homology group of any complex must be
torsion-free.  Meanwhile, $X_{\bar{S}}$ has the correct number of facets for a
$(d-1)$-CST of $X$, proving the ``if'' direction.
\end{proof}

\begin{proposition} \label{detD-formula}
If $\bd_{S,T}$ is nonsingular, then
  $$|\det \bd_{S,T}| = 
    \frac{|\HH_{d-1}(X_T)| \cdot |\HH_{d-2}(X_{\bar{S}})|}{ |\HH_{d-2}(X_T)|} = 
    \frac{|\HH_{d-1}(X_T)| \cdot |\HH_{d-2}(X_{\bar{S}})|}{ |\HH_{d-2}(X)|}.$$
\end{proposition}

\begin{proof}
As before, we interpret $\bd_{S,T}$ as the boundary map of the relative complex
$\Gamma=(X_T,X_{\bar{S}})$.  So $\bd_{S,T}$ is a map from $\Zz^{|T|}$ to
$\Zz^{|T|}$, and $\Zz^{|T|} / \bd_{S,T}(\Zz^{|T|})$ is a finite abelian group
of order $|\det\bd_{S,T}|$.  On the other hand, since $\Gamma$ has no faces
of dimension $\leq d-2$, its lower boundary maps are
all zero, so $|\det\bd_{S,T}| = |\HH_{d-1}(\Gamma)|$.  Since  $\HH_{d-2}(X_T)$ is finite,
the desired result now follows from the piece
  \begin{equation} \label{long-exact:2}
  0 \to \HH_{d-1}(X_T) \to \HH_{d-1}(\Gamma) \to \HH_{d-2}(X_{\bar{S}}) 
\to \HH_{d-2}(X_T) \to 0
  \end{equation}
of the long exact sequence \eqref{long-exact}.
\end{proof}

We now can state our main result connecting cellular spanning tree
enumeration with Laplacian eigenvalues.

\begin{theorem}[Cellular Matrix-Tree Theorem]
\label{thm:CMTT}
Let $d\geq 1$ and let $X^d$ be an APC cell complex.  Then:
\begin{enumerate}
\item We have
  $$\pi_d(X) = \frac{\tau_d(X) \tau_{d-1}(X)}{|\HH_{d-2}(X)|^2}.$$
\item Suppose that $d>0$.  Let $L=\Lud_{X,d-1}$,
let $U$ be the set of facets of a $(d-1)$-CST of $X$,
and let $L_U$ denote the reduced Laplacian obtained by deleting the rows
and columns of $L$ corresponding to $U$. 
Then 
  $$\tau_d(X) = \frac{|\HH_{d-2}(X)|^2}{|\HH_{d-2}(X_U)|^2} \det L_U.$$
\end{enumerate}
\end{theorem}

\begin{proof}[Proof of Theorem~\ref{thm:CMTT}~(1)]
The Laplacian $L$ is a $|X_{d-1}|$ by $|X_{d-1}|$ square matrix.  Since $X$ is APC, we have $\rank L=|X_d|-\betti_d=|X_d|-\betti_d+\betti_{d-1}$.
Let $\chi(L;y)=\det(yI-L)$ be
the characteristic polynomial of $L$ (where $I$ is an identity matrix of the same size), so that $\pi_d(X)$ is given up to sign by the coefficient
of $y^{|X_{d-1}|-|X_d|+\betti_d}$ in $\chi(L;y)$.  Equivalently,
  \begin{equation} \label{pi-formula}
  \pi_d ~= \sum_{\substack{S\subseteq X_{d-1}\\ |S|=\rank L}}\!\!\!\!\! \det L_U
        ~= \sum_{\substack{S\subseteq X_{d-1}\\ |S|=|X_d|-\betti_d}}\!\!\!\!\!\det L_U
  \end{equation}
where $U=X_{d-1}\sm S$ in each summand.  By the Binet--Cauchy formula, we have
  \begin{equation} \label{LT-formula}
  \det L_U ~= \sum_{\substack{T\subseteq X_d\\ |T|=|S|}}\!\!(\det\bd_{S,T})(\det\cbd_{S,T})
           ~= \sum_{\substack{T\subseteq X_d\\ |T|=|S|}}\!\!(\det\bd_{S,T})^2.
  \end{equation}
Combining \eqref{pi-formula} and \eqref{LT-formula}, applying
Proposition~\ref{nonsingular-criterion}, and interchanging the sums, we obtain
  $$\pi_d = \sum_{T:X_T\in\CST_d(X)}\ \ \sum_{S:X_{\bar{S}}\in\CST_{d-1}(X)} (\det\bd_{S,T})^2$$
and now applying Proposition~\ref{detD-formula} yields
  \begin{align*}
    \pi_d &=  \sum_{T:X_T\in\CST_d(X)}\ \ \sum_{S:X_{\bar{S}}\in\CST_{d-1}(X)}
    \left(\frac{|\HH_{d-1}(X_T)|\cdot|\HH_{d-2}(X_{\bar{S}})|}{|\HH_{d-2}(X)|}\right)^2\\\\
    &= \frac{ \left(\displaystyle\sum_{T:X_T\in\CST_d(X)} |\HH_{d-1}(X_T)|^2\right)
              \left(\displaystyle\sum_{S:X_{\bar{S}}\in\CST_{d-1}(X)} |\HH_{d-2}(X_{\bar{S}})|\right)}
            {|\HH_{d-2}(X)|^2}
  \end{align*}
as desired.
\end{proof}

In order to prove the ``reduced Laplacian'' part of Theorem~\ref{thm:CMTT},
we first need the following lemma.

\begin{lemma} \label{right-size}
Let $U$ be the set of facets of a $(d-1)$-CST of $X$, and let
$S = X_{d-1} \sm U$.
Then $|S|=|X_d|-\betti_d(X)$, the number of facets of a $d$-CST
of $X$.
\end{lemma}

\begin{proof}
Let $Y=X_{(d-1)}$; in particular,
  \begin{equation} \label{skeleton}
  |Y_\ell| = |X_\ell|\quad\text{for } \ell\leq d-1
  \qquad\text{and}\qquad
  \betti_\ell(Y) = \betti_\ell(X)\quad\text{for } \ell\leq d-2.
  \end{equation}
Therefore,
$|U|=|Y_{d-1}|-\betti_{d-1}(Y)+\betti_{d-2}(Y)=
|X_{d-1}|-\betti_{d-1}(Y)$, so $|S|=\betti_{d-1}(Y)$
by Proposition~\ref{two-out-of-three}.  Meanwhile, the Euler characteristics of
$X$ and $Y$ are
  \begin{align*}
  \chi(X) = \sum_{i=0}^d (-1)^i |X_i|     &= \sum_{i=0}^d (-1)^i \betti_i(X),\\
  \chi(Y) = \sum_{i=0}^{d-1} (-1)^i |Y_i| &= \sum_{i=0}^{d-1} (-1)^i \betti_i(Y).
  \end{align*}
By \eqref{skeleton}, we see that
  $$
  \chi(X)-\chi(Y)
    ~=~ (-1)^d |X_d|
    ~=~ (-1)^d \betti_d(X) + (-1)^{d-1}\betti_{d-1}(X) - (-1)^{d-1}\betti_{d-1}(Y)
  $$
from which we obtain
  $
  |X_d| = \betti_d(X) - \betti_{d-1}(X) + \betti_{d-1}(Y).
  $
Since $X$ is APC, we have $\betti_{d-1}(X)=0$, so
$|S| =\betti_{d-1}(Y) = |X_d| - \betti_d(X)$ as desired.
\end{proof}

\begin{proof}[Proof of Theorem~\ref{thm:CMTT}~(2)]
By the Binet--Cauchy formula, we have
$$
\det L_U = \sum_{T:\ |T|=|S|} (\det\bd_{S,T})(\det\cbd_{S,T})
= \sum_{T:\ |T|=|S|} (\det\bd_{S,T})^2.
$$

By Lemma~\ref{right-size} and Proposition~\ref{nonsingular-criterion}, $\bd_{S,T}$ is nonsingular exactly
when $X_T\in\CST_d(X)$.  Hence Proposition~\ref{detD-formula} gives
  \begin{align*}
  \det L_U &= \sum_{T:X_T \in \CST_d(X)}
             \left(\frac{|\HH_{d-1}(X_T)|\cdot|\HH_{d-2}(X_U)|}
             {|\HH_{d-2}(X)|}\right)^2 \\
           &= \frac{|\HH_{d-2}(X_U)|^2}{|\HH_{d-2}(X)|^2}
             \sum_{T:X_T\in\CST_d(X)} |\HH_{d-1}(X_T)|^2
           ~=~ \frac{|\HH_{d-2}(X_U)|^2}{|\HH_{d-2}(X)|^2} \tau_d(X),
  \end{align*}
which is equivalent to the desired formula.
\end{proof}

We will often work with complexes that are in fact $\Zz$-acyclic
in positive codimension (as opposed to merely $\Qq$-acyclic), and
whose Laplacians have nice forms.  In this case, the following formula
is often the most convenient way to obtain tree enumerators from
Laplacian eigenvalues.

\begin{corollary}
\label{alternating-product}
Let $X^d$ be a cell complex such that $\HH_i(X,\Zz)=0$ for all $i<d$.  Then,
for every $k\leq d$, we have
  $$\tau_k ~=~ \prod_{i=0}^k \pi_i^{(-1)^{k-i}}.$$
\end{corollary}

\begin{proof}
Applying Theorem~\ref{thm:CMTT}~(1) repeatedly, we obtain
$$\tau_k ~=~ \frac{\pi_k}{\tau_{k-1}} ~=~ \frac{\pi_k}{\pi_{k-1}}\tau_{k-2} ~=~ \cdots
~=~ \left(\prod_{i=1}^k \pi_i^{(-1)^{k-i}}\right)\tau_0^{(-1)^k}
~=~ \prod_{i=0}^k \pi_i^{(-1)^{k-i}}.\eqno\qedhere
$$
\end{proof}

For later use (when we study duality in Section~\ref{sec:duality}), define
  $$\omega_k = \omega_k(X) = \prod_{0\neq\lambda\in\stot_k(X)} \lambda.$$
Note that equation~\eqref{tot-from-uddu} implies that
$\omega_k = \pi_k\pi_{k+1}$ for $k>0$.  Moreover, $\omega_0=\pi_1$.
 Solving for the $\pi$'s in terms of
the $\omega$'s gives
  $$\pi_i = \prod_{j=0}^{i-1} \omega_j^{(-1)^{i-j+1}}$$
and substituting this formula into Corollary~\ref{alternating-product} gives
\begin{align}
\tau_k
&~=~ \prod_{i=0}^k \prod_{j=0}^{i-1} \omega_j^{(-1)^{k-j+1}}
~=~ \prod_{j=0}^{k-1} \prod_{i=j+1}^k \omega_j^{(-1)^{k-j+1}}
~=~ \prod_{j=0}^{k-1} \omega_j^{\sum_{i=j+1}^k (-1)^{k-j+1}}\notag\\
&~=~ \prod_{j=0}^{k-1} \omega_j^{ (-1)^{k-j+1}(k-j)}.\label{tau-from-omega}
\end{align}

\subsection{The Weighted Cellular Matrix-Tree Theorem}
\label{WCMTT-section}

As before, let $d\geq 1$ and let $X^d$ be a cell complex that is APC.
Introduce an indeterminate $x_f$ for each $f\in X_d$, and let $X_f=x_f^2$.
For every $T\subseteq X_d$, let $x_T = \prod_{f\in T} x_f$ and let $X_T=x_T^2$.
To construct the \emph{weighted boundary matrix} $\wbd_{X,d}$, we
multiply each column of $\bd_{X,d}$ by $x_f$, where $f$ is the corresponding $d$-cell of $X$.
We can accordingly define weighted versions of the coboundary maps, Laplacians,
etc., of Section~\ref{CMTT-section}, as well as of the
invariants $\pi_k$ and $\tau_k$.  We will notate each weighted object by placing
a hat over the symbol for the corresponding unweighted quantity.  Thus
$\hat\pi_k$ is the product of the nonzero eigenvalues of $\Lwud_{X,k-1}$
(for $k\geq 1$), and
  $$\hat \tau_k = \hat \tau_k(X) = \sum_{Y\in\CST_k(X)} |\HH_{k-1}(Y)|^2 X_Y.$$
Meanwhile, any unweighted quantity can be recovered from its weighted analogue simply by setting $x_f=1$ for all $f\in X_d$.

\begin{proposition} \label{weighted-tools}
Let $T\subseteq X_d$ and $S\subseteq X_{d-1}$, with $|T|=|S|=|X_d|-\betti_d$.
Then $\det{\wbd}_{S,T} = x_T \det \bd_{S,T}$ is nonzero
if and only if $X_T\in\CST_d(X)$ and $X_{\bar S}\in\CST_{d-1}(X)$.  In that case,
  \begin{equation} \label{weighted-detD}
  \pm\det\wbd_{S,T} =
  \frac{|\HH_{d-1}(X_T)| \cdot |\HH_{d-2}(X_{\bar S})|}{ |\HH_{d-2}(X_T)|} x_T  = 
  \frac{|\HH_{d-1}(X_T)| \cdot |\HH_{d-2}(X_{\bar S})|}{ |\HH_{d-2}(X)|} x_T.
  \end{equation}
\end{proposition}

\begin{proof}
The first claim follows from Proposition~\ref{nonsingular-criterion}, and the second
from Proposition~\ref{detD-formula}.
\end{proof}

It is now straightforward to adapt the proofs of both parts of
Theorem~\ref{thm:CMTT} to the weighted setting.

\begin{theorem}
\label{thm:WCMTT}
Let $d \geq 1$, let $X^d$ be an APC cell complex,
and let $\hat L=\Lwud_{X,d-1}$.  
\begin{enumerate}
\item We have
  $$\hat\pi_d(X) = \frac{\hat \tau_d(X) \tau_{d-1}(X)}{|\HH_{d-2}(X)|^2}.$$
\item Let $U$ be the set of facets of a $(d-1)$-CST of $X$, and let $\hat L_U$ be the reduced Laplacian obtained by
deleting the rows and columns of $\hat L$ corresponding to $U$.  Then
  $$\hat \tau_d(X) = \frac{|\HH_{d-2}(X)|^2}{|\HH_{d-2}(X_U)|^2} \det \hat L_U.$$
\end{enumerate}
\end{theorem}

\section{Cubical complexes}
\label{sec:cubical}

We now specialize from arbitrary cell complexes to cubes and
cubical complexes.  We retain the notation and terminology of
Section~\ref{cellular-section} for cell complexes.

The \emph{$n$-cube} $Q_n$ is the face poset of the geometric $n$-cube $\tilde Q_n$,
the convex hull of the $2^n$ points in $\Rr^n$ whose coordinates are all $0$ or $1$.
We will usually blur the distinction between the polytope
$\tilde Q_n$ and its face poset $Q_n$.  The faces of $Q_n$ are ordered
$n$-tuples $f=(f_1,\dots,f_n)$, where $f_i\in\{0,1,\X\}$.
Intuitively, we think of $\X$ as the open interval $(0,1)$.
For example, the cell $(0,\X,1,\X)\in Q_4$ corresponds to the 2-cell
  $$\{0\} ~\x~ (0,1) ~\x~ \{1\} ~\x~ (0,1) ~\subset~ \tilde Q_4 ~\subset~\Rr^4.$$
The order relation in $Q_n$ is as follows:
$f\leq g$ iff $f_i\leq g_i$ for all $i\in[n]$,
where $0<\X$; $1<\X$; and $0,1$ are
incomparable.
(If necessary, we can regard $Q_n$ as containing a unique minimal element $\0$,
with undefined direction and dimension~$-1$.)

The \emph{direction} of a face~$f$ is defined as
$\dir(f) = \{i\in[n] \st f_i=\X\}$, so that
$\dim(f) = |\dir(f)|$.
Notice that $\dir(f)\subseteq\dir(g)$ whenever $f\leq g$,
although the converse is not true.
The poset $Q_n$ is ranked, with $\binom{n}{i}2^{n-i}$
faces of rank~$i$ for $0\leq i\leq n$.

A \emph{proper cubical complex} $X$ is an order ideal in $Q_n$.
This is a combinatorial object with a natural geometric realization $\tilde X$
as the union of the corresponding faces of the polytope $\tilde Q_n$.
Note that this is a much more restrictive definition of ``cubical complex''
than as simply a cell complex all of whose faces are combinatorial cubes.
The reason for working with this smaller class of cubical complexes
is that the cells and boundary maps can be described combinatorially, as we now
explain.

Let $f$ and $g$ be faces of $Q_n$ of dimensions $i-1$ and $i$
respectively.  If $f\leq g$, then we may write
  $\dir(g)=\{a_1,\dots,a_i\}$, $\dir(f)=\{a_1,\dots,\widehat{a_j},\dots,a_i\}$,
with both direction sets listed in increasing order.  Then the
relative orientation of the pair $f,g$ is
  $$\sign(f,g) = \begin{cases}
    (-1)^j & \text{ if } f_{a_j}=0,\\
    (-1)^{j+1} & \text{ if } f_{a_j}=1.
  \end{cases}$$
Meanwhile, if $f\not\leq g$, then we set $\sign(f,g)=0$.

Let $R$ be a coefficient ring (typically $\Zz$ or a field),
and let $C_i(X)=C_i(X,R)$ be the free abelian group 
on generators $[f]$ for $f\in X_i$.  The
\emph{$i^{th}$ cubical boundary map} of $X$ is
  \begin{align*}
  \bd_{X,i}:\ C_i(X) &\to C_{i-1}(X)\\
  [g] &\mapsto \sum_{f\in X_{i-1}} \sign(f,g) [f]
  \end{align*}
and the \emph{$i^{th}$ cubical coboundary map} of $X$ is
  \begin{align*}
  \cbd_{X,i}:\ C_{i-1}(X) &\to C_i(X)\\
  [f] &\mapsto \sum_{g\in X_i} \sign(f,g) [g].
  \end{align*}
As before, we define the \emph{$i^{th}$ up-down}, \emph{down-up}, and \emph{total cubical Laplacians} as
respectively
  $$
  \Lud_{X,i} = \bd_{X,i+1} \cbd_{X,i+1},\qquad
  \Ldu_{X,i} = \cbd_{X,i} \bd_{X,i},\qquad
  \Ltot_{X,i} = \Lud_{X,i}+\Ldu_{X,i}.
  $$

The map $\bd_{X,i}$ is in fact the cellular boundary map of $Q_n$ as a
cell complex; see \cite[\S4]{EhrHet}.  So the techniques of
Section~\ref{cellular-section} can be applied to count cubical
spanning trees. 

\begin{example}
\label{q1-unweighted}
A fundamental example is the complex $X=Q_1$,
with $X_0=\{0,1\}$, $X_1=\{\X\}$.  The
boundary, coboundary and Laplacian matrices of $X$ are
  \begin{align*}
  \bd_1 &= \begin{bmatrix} -1\\1\end{bmatrix}, & 
  \Lud_0 = \Ltot_0 &= \begin{bmatrix} 1 & -1\\-1 & 1\end{bmatrix},\\
  \cbd_1 &= \begin{bmatrix} -1 & 1\end{bmatrix}, &
  \Ldu_1 = \Ltot_1 &= \begin{bmatrix} 2\end{bmatrix},
  \end{align*}
and so its spectrum polynomial is
  \begin{equation} \label{spec:Q1}
  \Etot(Q_1;\q,t) = \sum_{i=0}^{\dim X}\sum_{\lambda\in\stot_i(X)} \!\!\! t^i \q^\lambda
    = 1+\q^2+t\q^2.
  \end{equation}
\end{example}

\subsection{Prisms}
\label{sec:prism}
We now consider the important \emph{prism} operation, which is the
cubical analogue of coning a simplicial complex.

\begin{definition} \label{defn:prism}
Let $X\subseteq Q_n$ be a proper cubical complex.  The \emph{prism}
over $X$ is the subcomplex of $Q_{n+1}$ defined by
  $$\Prism X = \left\{f=(f_1,\dots,f_{n+1})\in Q_{n+1}\st (f_1,\dots,f_n)\in X\right\}.$$
\end{definition}

Note that
  \begin{align*}
  (\Prism X)_i &= \left\{(f_1,\dots,f_n,0),\ (f_1,\dots,f_n,1)\st (f_1,\dots,f_n)\in X_i\right\}\\
  &\cup \left\{f=(f_1,\dots,f_n,\X)\st (f_1,\dots,f_n)\in X_{i-1}\right\}.
  \end{align*}

As a cell complex, $\Prism X$ is just the product $X\x Q_1$.  Therefore,
we can use the product formula, Theorem~\ref{product-laplacian}, to
write down the Laplacian spectra of $\Prism X$ in terms of those of $X$.
The formula can be stated in several equivalent ways, all
of which will be useful in different contexts.

\begin{theorem}
\label{prism-spectrum}
Let $X^d$ be a proper cubical complex.  Then

\begin{subequations}
\begin{align}
\label{prism:Ert}
\Etot(\Prism X;\q,t) &= (1+\q^2+t\q^2) \Etot(X;\q,t),\\
\label{prism:Er}
\Etot_i(\Prism X;\q) &= (1+\q^2)\Etot_i(X;\q) + \q^2\Etot_{i-1}(X;\q),\\
\label{prism:stot}
\stot_i(\Prism X) &= \{\lambda\st\lambda\in \stot_i(X)\}
  ~\cup~ \{\lambda+2\st\lambda\in \stot_i(X)\}
  ~\cup~ \{\mu+2\st\mu\in \stot_{i-1}(X)\},\\
\label{prism:sud}
\sud_i(\Prism X) &\zeq \{\lambda\st\lambda\in \sud_i(X)\}
  ~\cup~ \{\lambda+2\st\lambda\in \stot_i(X)\},
\end{align}
\end{subequations}

\noindent for all $0\leq i\leq d$, where $\cup$ denotes multiset union.

In particular, the prism operation preserves Laplacian integrality.
\end{theorem}

\begin{proof}
Equation~\eqref{prism:Ert} follows from Theorem~\ref{product-laplacian}
together with equation~\eqref{spec:Q1}, and \eqref{prism:Er} and \eqref{prism:stot} are just rephrasings of~\eqref{prism:Ert}.

To prove~\eqref{prism:sud}, we proceed by induction on $i$.  
For $i=0$, the formula follows from \eqref{prism:stot}, because
$\sud_0(X) = \stot_0(X)$ and $\stot_{-1}(X)=\0$.
For $i\geq 1$, we have
\begin{align*}
\sud_i(\Prism X) &\zeq \stot_i(\Prism X) ~\sm~ \sud_{i-1}(\Prism X)\\
&= \stot_i(\Prism X) ~\sm~ (\sud_{i-1}(X) ~\cup~\{\lambda+2\st\lambda\in \stot_{i-1}(X)\})\\
&\zeq \sud_i(X) ~\cup~ \{\lambda+2\st\lambda\in \stot_{i}(X)\} ~\cup~ \sud_{i-1}(X) ~\cup~ \{\lambda+2\st\lambda\in \stot_{i-1}(X)\}\\
&\qquad\sm~(\sud_{i-1}(X) ~\cup~ \{\lambda+2\st\lambda\in \stot_{i-1}(X)\})\\
&= \sud_i(X) ~\cup~ \{\lambda+2\st\lambda\in \stot_{i}(X)\}.
\end{align*}
\end{proof}

\subsection{Laplacian spectra of cubes}

As a consequence of Theorem~\ref{prism-spectrum}, we obtain a formula for
the Laplacian eigenvalues of $Q_n$, and thus for the torsion-weighted
spanning tree enumerators $\tau_k(Q_n)$.

\begin{theorem}
\label{cube-eigenvalue-theorem}
Cubes and their skeletons are Laplacian integral.  Specifically,
for all $n\geq 1$, we have
\begin{align}
 \Etot(Q_n;\q,t) &= (1+\q^2+t\q^2)^n ~=~ \sum_{k=0}^n t^k \binom{n}{k} \q^{2k} (1+\q^2)^{n-k},\\
\label{ud-du-cube-spectrum}
 \Eud(Q_n;\q,t) 
 &=~ \sum_{k=0}^{n-1} t^k \left[\sum_{j=k+1}^n \binom{n}{j}\binom{j-1}{k} \q^{2j}\right].
\end{align}
\end{theorem}

\begin{proof}
The formula for $\Etot_k(Q_n;\q,t)$ follows from Theorem~\ref{prism-spectrum},
since $Q_n$ can be identified with the $n$-fold product $Q_1\x\cdots\x Q_1$ as cubical
complexes (indeed, as cell complexes).

By \eqref{ud-recurrence},
we can obtain $\Eud(Q_n;\q,t)$ from $\Etot(Q_n;\q,t)$ by deleting all the
$\q$-free terms (i.e., those which correspond to zero eigenvalues) and
dividing by $1+t$.  The only such term is a single $1$ (from the $k=0$
summand).  Therefore,
  \begin{align*}
  \Eud(Q_n;\q,t) &= \frac{\left(1+\q^2(1+t)\right)^n-1}{1+t}\\
    &=~ \frac{\left(\displaystyle\sum_{j=0}^n\binom{n}{j} \q^{2j} (1+t)^j \right)-1}{1+t}
    ~=~ \frac{\displaystyle\sum_{j=1}^n\binom{n}{j} \q^{2j} (1+t)^j}{1+t}\\
    &=~ \sum_{j=1}^n\binom{n}{j} \q^{2j} (1+t)^{j-1}
    ~=~ \sum_{j=1}^n\binom{n}{j} \q^{2j} \sum_{k=0}^{j-1} \binom{j-1}{k} t^k\\
    &=~ \sum_{k=0}^{n-1} t^k \left[\sum_{j=k+1}^n \binom{n}{j}\binom{j-1}{k} \q^{2j}\right].
  \end{align*}
\end{proof}

\begin{corollary}
\label{thm:general-count}
Let $n\geq 1$ and $1\leq k\leq n$.  Then
  $$\tau_k(Q_n) = \prod_{j=k+1}^n (2j)^{\binom{n}{j} \binom{j-2}{k-1}}.$$
\end{corollary}

\begin{proof}
Theorem~\ref{cube-eigenvalue-theorem} implies that
  $$\pi_i(Q_n) = \prod_{j=i}^n (2j)^{\binom{n}{j}\binom{j-1}{i-1}}
  = \prod_{j=1}^n (2j)^{\binom{n}{j}\binom{j-1}{i-1}}$$
for $0<i\leq n$ (adopting the convention that $\binom{a}{b}=0$
for $a<b$).  Applying the alternating product formula, Corollary~\ref{alternating-product},
gives
\begin{align*}
\tau_k(Q_n) ~&=~ \prod_{i=0}^k \pi_i^{(-1)^{k-i}}
~=~ 2^{n(-1)^k}\prod_{i=1}^k \left[ \prod_{j=1}^n (2j)^{\binom{n}{j}\binom{j-1}{i-1}} \right]^{(-1)^{k-i}}\\
&=~ 2^{n(-1)^k}\prod_{j=1}^n (2j)^{\binom{n}{j} \left(\sum_{i=1}^k (-1)^{k-i} \binom{j-1}{i-1} \right)}.
\end{align*}
The $j=1$ factor in this product simplifies to $2^{n(-1)^{k-1}}$,
canceling the initial factor of $2^{n(-1)^k}$.
For $2\leq j\leq k$, the sum in the exponent vanishes,
and for $k+1\leq j\leq n$, it simplifies to
$\binom{j-2}{k-1}$ (this can be seen by
repeatedly applying Pascal's identity), giving the desired formula.
\end{proof}


\section{Weighted Laplacians of cubes}
\label{sec:weighted}

In this section, we study a weighting that associates a Laurent monomial
to each face of $Q_n$, giving finer information about Laplacian spectra of cubes.

\subsection{Algebraically weighted eigenvalues}

Let $X\subseteq Q_n$ be a proper cubical complex, and introduce commuting indeterminates
$x_i,y_i,q_i$ for $i\in\Nn$.  Weight each face
$f=(f_1,\dots,f_n)\in Q_n$ by the monomial
  $$\xi_f = \prod_{i:f_i=\X} q_i \prod_{i:f_i=0} x_i \prod_{i:f_i=1} y_i.$$

Define the \emph{algebraically weighted cubical boundary map}
by
  \begin{equation} 
\label{alg-cube-bd}
  \begin{aligned}
  \wbd_{X,k}\colon C_k(X) &\to C_{k-1}(X)\\
  [g] &\mapsto \sum_{f\in X_{k-1}} \sign(f,g) \frac{\xi_g}{\xi_f} [f]
  \end{aligned}
  \end{equation}
so that the corresponding weighted coboundary map is
  \begin{equation}
 \label{alg-cube-cobd}
  \begin{aligned}
  \wcbd_{X,k}\colon C_{k-1}(X) &\to C_k(X)\\
  [f] &\mapsto \sum_{g\in X_k} \sign(f,g) \frac{\xi_g}{\xi_f} [g].
  \end{aligned}
  \end{equation}
It is easy to check that $\wbd\wbd=\wcbd\wcbd=0$.
(This vital equality would fail if we had defined the weighted
boundary more ``combinatorially naturally'' by
$[g]\mapsto\sum_{f\in X_{k-1}} \sign(f,g)\xi_g[f]$.)

The \emph{$k^{th}$ up-down}, \emph{down-up}, and \emph{total algebraically
weighted cubical Laplacians} are
respectively
  $$
  \Lwud_{X,k} = \wbd_{X,k+1} \wcbd_{X,k+1},\qquad
  \Lwdu_{X,k} = \wcbd_{X,k} \wbd_{X,k},\qquad
  \Lwtot_{X,k} = \Lwud_{X,k}+\Lwdu_{X,k}.
  $$
As in the unweighted case, it is convenient to record the eigenvalues
as generating functions:
\begin{align*}
\Ewany_k(X;\q) &= \sum_{\lambda\in \Swany_k(X)} \!\!\! \q^\lambda,\\
\Ewany(X;\q,t) &= \sum_{k=0}^{\dim X}\sum_{\lambda\in\Swany_k(X)} \!\!\! t^k \q^\lambda,
\end{align*}
where $\bullet\in\{\text{ud},\text{du},\text{tot}\}$. 

\begin{example}
\label{q1-weighted}
Consider the complex $X=Q_1$ (see Example~\ref{q1-unweighted}),
whose edge we regard as lying in direction $i$. The vertices have weights $x_i$ and $y_i$, and the edge has weight $q_i$.  The
weighted boundary, coboundary and Laplacian matrices are thus
  \begin{align*}
  \wbd_1 &= \begin{bmatrix} -q_i/x_i\\q_i/y_i\end{bmatrix}, & 
  \Lwud_0 = \Ltot_0 &= \begin{bmatrix} q_i^2/x_i^2 & -q_i^2/x_iy_i\\ -q_i^2/x_iy_i & q_i^2/y_i^2 \end{bmatrix},\\
  \wcbd_1 &= \begin{bmatrix} -q_i/x_i & q_i/y_i\end{bmatrix}, &
  \Lwdu_1 = \Ltot_1 &= \begin{bmatrix} q_i^2/x_i^2+q_i^2/y_i^2\end{bmatrix},
  \end{align*}
and so the weighted spectrum polynomial is
  \begin{equation} \label{spec:Q1-weighted}
  \Ewtot(Q_1;\q,t) = \sum_i\sum_{\lambda\in\stot_i(Q_1)} \!\!\! t^i \q^\lambda
    = 1+\q^{u_i}+t\q^{u_i}
  \end{equation}
where
  $$u_i=\frac{q_i^2}{x_i^2}+\frac{q_i^2}{y_i^2}.$$
\end{example}

Just as in the unweighted setting, the total eigenvalue
generating function $\Ewtot(X;\q,t)$ is multiplicative on products of
cubical complexes, that is,
  \begin{equation} \label{weighted-product-formula}
  \Ewtot(X\x Y;\q,t) = \Ewtot(X;\q,t)\Ewtot(Y;\q,t).
  \end{equation}
This formula is proved in exactly the same way as its unweighted
analogue Theorem~\ref{product-laplacian} (which can be recovered
from~\eqref{weighted-product-formula}
by setting $q_i=x_i=y_i=1$).
In particular, if $X\subseteq Q_{n-1}$ is a proper cubical complex,
then the prism $\Prism X\subseteq Q_n$ is just the product of $X$
with a copy of $Q_1$ lying in direction~$n$.  Hence
  \begin{equation} \label{wspec:prism}
  \Ewtot(\Prism X;\q,t) ~=~ \left(1+\q^{u_n}+t\q^{u_n}\right) \Ewtot(X;\q,t).
  \end{equation}
Note that specializing $q_i=x_i=y_i=1$ makes $u_i=2$, and so recovers
the first assertion of Theorem~\ref{prism-spectrum}.  The corresponding
recurrence for the multisets of eigenvalues is
  \begin{equation} \label{weighted-prism-spectrum}
  \Swtot_i(\Prism X) ~=~ \{\lambda\st\lambda\in \Swtot_i(X)\}
  ~\cup~\left\{\lambda+u_n\st\lambda\in \Swtot_i(X)\right\}
  ~\cup~\left\{\mu+u_n\st\mu\in \Swtot_{i-1}(X)\right\}.
  \end{equation}

We now calculate the eigenvalues of the weighted Laplacians of
the full cube $Q_n$.

\begin{theorem}
\label{cube-weighted-eigenvalue-theorem}
For all $n\geq 1$, we have
\begin{align}
  \Ewtot(Q_n;\q,t) &= \prod_{k=1}^n \left(1+\q^{u_k}+t\q^{u_k}\right) ~=~
  \sum_{i=0}^n t^i \left[\sum_{k=i}^n \binom{k}{i} e_k\right], \label{CWE:1}\\
  \Ewud(Q_n;\q,t) &= \sum_{j=1}^n (1+t)^{j-1} e_j ~=~
  \sum_{i=0}^{n-1} t^i \left[\sum_{j=i+1}^n \binom{j-1}{i} e_j\right], \label{CWE:2}
\end{align}
where $e_j$ denotes the $j^{th}$ elementary symmetric function in
the forms $\q^{u_1},\dots,\q^{u_n}$.
\end{theorem}

\begin{proof}
The first equality of \eqref{CWE:1} follows from iterating
the weighted product formula \eqref{weighted-product-formula}.
Extracting the $t^i$ coefficient to find $\Etot_i(Q_n;\q)$
and abbreviating $u_A=\sum_{i\in A} u_i$ for $A\subseteq[n]$
we have
\begin{align*}
\Ewtot_i(Q_n;\q)
  &= \sum_{\substack{A\subseteq[n]\\ |A|=i}} \left(\prod_{j\in A} r^{u_j}\right)
       \left(\prod_{j\not\in A} (1+r^{u_j})\right)
  \quad=\quad \sum_{\substack{A\subseteq[n]\\ |A|=i}} r^{u_A} \ \ \sum_{B\subseteq[n]\sm A} r^{u_B}\\
  &= \sum_{k=i}^n \sum_{\substack{C\subseteq[n]\\ |C|=k}} \ \ \sum_{\substack{A\subseteq C\\|A|=i}} r^{u_C}
  \quad=\quad \sum_{k=i}^n \ \ \sum_{\substack{C\subseteq[n]\\ |C|=k}} \binom{k}{i} r^{u_C}\\
  &= \sum_{k=i}^n \binom{k}{i} e_k
\end{align*}
which is the second equality of \eqref{CWE:1}.

Finally, by the weighted analogues of the equalities \eqref{ud-recurrence},
we can obtain $\Ewud(Q_n;\q,t)$ from $\Ewtot(Q_n;\q,t)$ by deleting all the
$\q$-free terms (i.e., those which correspond to zero eigenvalues) and
dividing by $1+t$:
  \begin{align*}
  \Ewud(Q_n;\q,t) &= \frac{\left(\displaystyle\prod_{k=1}^n 1+(1+t)\q^{u_k}\right)-1}{1+t}\\
   &=~ \sum_{j=1}^n (1+t)^{j-1} e_j
   ~=~ \sum_{j=1}^n \sum_{i=0}^{j-1} \binom{j-1}{i} t^i e_j\\
   &=~ \sum_{i=0}^{n-1} t^i \left[\sum_{j=i+1}^n \binom{j-1}{i} e_j\right].
  \end{align*}
\end{proof}

Extracting the $t^i$ coefficient from formula \eqref{CWE:2}
of Theorem~\ref{cube-weighted-eigenvalue-theorem} gives
  $$\Ewud_i(Q_n;\q) = \sum_{\substack{A\subseteq[n]\\ |A|\geq i+1}} \binom{|A|-1}{i} \q^{u_A}.$$
That is, the nonzero eigenvalues
of $\Lwud_i(Q_n)$ are the expressions $u_A$, each occurring with multiplicity
$\binom{|A|-1}{i}$.  Therefore, the product of nonzero eigenvalues---that
is, the algebraically weighted analogue of the invariant $\pi_i(Q_n)$---is
  \begin{equation} \label{product-alg-NZE}
  \prod_{\substack{\lambda\in\Swud_i(Q_n)\\ \lambda\neq 0}} \lambda = \prod_{\substack{A\subseteq[n]\\ |A|\geq i+1}} u_A^{\binom{|A|-1}{i}}
  = \prod_{\substack{A\subseteq[n]\\ |A|\geq i+1}} \left[\sum_{j\in A} \left(\frac{q_j^2}{x_j^2}+\frac{q_j^2}{y_j^2}\right)\right]^{\binom{|A|-1}{i}}.
  \end{equation}

\subsection{Weighted tree enumeration}

We now consider the polynomial generating function
  $$\hat\tau_k(X) = \sum_{Y\in\CST_k(X)} \prod_{g\in Y_k}\xi_g$$
for a proper cubical complex $X$.  As usual, the main case of interest
is $X=Q_n$.  In principle, the invariants $\hat\tau_k(Q_n)$
can be computed in terms of the eigenvalues of weighted Laplacians,
using Theorem~\ref{thm:WCMTT}.  Those ``combinatorially weighted''
Laplacians look similar to (in fact, simpler than) the algebraically
weighted Laplacians discussed in the previous section, but they
do not come from well-defined boundary maps (i.e., whose square is zero)
and their eigenvalues are not even polynomials.  On the other hand,
there is strong evidence for the following formula.

\begin{conjecture}
\label{conj:new}
$$\hat\tau_k(Q_n) =
    (q_1\cdots q_n)^{\displaystyle \sum_{i=k-1}^{n-1}\binom{n-1}{i}\binom{i-1}{k-2}}
    \prod_{\substack{A\subseteq[n]\\|A|\geq k+1}} \left[
    \sum_{i\in A} \left(q_i(x_i+y_i)\prod_{j\in A\sm i} x_j y_j \right) \right]^{\displaystyle{\binom{|A|-2}{k-1}}}.$$
\end{conjecture}

The conjectured formula is similar to equation~\eqref{product-alg-NZE}:
specifically, clearing denominators from the square-bracketed
expression in~\eqref{product-alg-NZE} indexed by~$A\subseteq[n]$
gives the corresponding factor indexed by $A$ in the right-hand side
of the conjecture.  Our original goal in proving formulas such as
Theorem~\ref{cube-weighted-eigenvalue-theorem} was to prove
Conjecture~\ref{conj:new} by translating between the algebraically
and combinatorially weighted Laplacians; this approach had succeeded
in the case of shifted simplicial complexes \cite{DKM}, but
it is not clear how to do that here.

Conjecture~\ref{conj:new} can be verified computationally for small values of
$n$ and $k$.  The case $k=1$ is Theorem~3 of \cite{MR}, and
the case $k=n-1$ can be checked directly, because an $(n-1)$-spanning
tree of $Q_n$ is just a subcomplex generated by all but one of its $(n-1)$-faces.

We suggest a possible avenue for proving  Conjecture~\ref{conj:new}.
First, note that the weighted spanning tree enumerator is homogeneous in the
variable sets $\{x_1,y_1,q_1\},\dots,\{x_n,y_n,q_n\}$, so we lose
no information by setting $q_1=\cdots=q_n=1$ on the right-hand
side, obtaining the simpler formula
$$F(n,k) = \prod_{\substack{A\subseteq[n]\\|A|\geq k+1}} \left[
    \sum_{i\in A} \left((x_i+y_i)\prod_{j\in A\sm i} x_j y_j \right) \right]^{\displaystyle\binom{|A|-2}{k-1}}.$$
Second, we observe that $F(n,k)$ satisfies the recurrence
\begin{equation*}
\label{STrecurrence}
F(n,k) = G([n],n-1,k) F(n,n-1)^{\binom{n-2}{k-1}}
\end{equation*}
where
$$G(S,a,b) = \prod_{A \subseteq S, |A| = a} F(A,b).$$ 

\section{Shifted cubical complexes}
\label{sec:shifted}

Inspired by the notion of shifted simplicial complexes (see, e.g.,
\cite{Kalai_shifting,DR}), we define
the class of shifted cubical complexes.  These shifted cubical
complexes share many properties with their simplicial
counterparts.  In particular, they are Laplacian integral,
constructible from a few basic operations, and arise as order ideals with
respect to a natural relation on direction sets.
The guiding principle
is that \emph{directions} in cubical complexes are analogous to
\emph{vertices} in simplicial complexes.

We first generalize our definitions slightly to be able to work with
cubical complexes on arbitrary direction sets.  That is, if $D$ is
any set of positive integers, a \emph{cubical complex with direction set $D$}
is a family $X$ of ordered tuples $(f_i)_{i\in D}$, where $f_i\in\{0,1,\X\}$ for all~$i$,
closed under replacing $\X$'s with 0's or with 1's.  The direction and dimension of
faces are defined as before: $\dir(f) = \{i\in D \st f_i=\X\}$, and
$\dim f=|\dir(f)|$.  We will frequently need to regard the direction of a face as
a list in increasing order; in this case we write $\dir(f)_<$ instead of
$\dir(f)$.

\begin{definition}
A \emph{shifted cubical complex} is a proper cubical complex
$X\subseteq Q_n$ that satisfies
the following conditions for every $f,g \in \{0,1,\X\}^n$
with $\dim f=\dim g$:

\begin{enumerate}
\item\label{fullskel} $X$ contains the full 1-skeleton of $Q_n$.

\item\label{dironly} If $f \in X$ and $\dir(f) = \dir(g)$, then $g \in X$.

\item\label{cubicalshifted}
If $g \in X$ and $\dir(f)_<$ precedes $\dir(g)_<$ in
component-wise partial order (that is, the $i^{th}$ smallest element
of $\dir(f)_<$ is less than or equal to that of $\dir(g)_<$
for every $i$), then $f \in X$.
\end{enumerate}
\end{definition}

The first condition is analogous to the requirement that a simplicial complex
``on vertex set $V$'' actually contain each member of $V$ as a vertex.
In light of condition~\eqref{dironly}, it would be equivalent to replace~\eqref{fullskel} with the condition that $X$ contains \emph{at least} one edge in every possible direction.
The second condition reflects a symmetry between the digits 0 and 1,
while the last condition is the cubical analogue
of the definition of a shifted simplicial complex (as a complex whose faces,
regarded as collections of vertices, form
an order ideal in component-wise partial order).

\subsection{Near-prisms}
Bj\"orner and Kalai \cite{BK} introduced the concept of near-cones, computed their homotopy type, and showed that shifted simplicial complexes are near-cones.
The cubical counterpart of a near-cone is a near-prism.  We develop the
basic facts about near-prisms in parallel to the presentation
of~\cite[Section 5]{DR}, in order to prove that shifted cubical complexes are Laplacian integral.

Throughout this section, let $X\subseteq\{0,1,\X\}^n$ be a
$d$-dimensional proper cubical complex.
For a direction $i\in[n]$, define the deletion and link with
respect to $i$ as follows:
  \begin{align*}
  \Del{i}{X}  &= \{f\backslash f_i\st f\in X,\ f_i\neq\X\},\\
  \Link{i}{X} &= \{f\backslash f_i\st f\in X,\ f_i =\X\}.
  \end{align*}

The deletion and link in direction $i$ are proper cubical complexes with
direction set $[n]\backslash i$.  Meanwhile, given a complex $X$ on
direction set $[n]\backslash i$, define the \emph{prism of $X$ in
direction $i$} as follows:
$$\Prism_i X = \{(f_1,\dots, f_n) \in \{0,1,\X\}^n
\st f\backslash f_i \in X\}.$$
If $X\subseteq Q_{n-1}$, then the prism $\Prism X$ defined in Section~\ref{sec:prism}
is $\Prism_n X$ in this notation.

Observe that $\Prism_i X$ naturally contains two isomorphic copies of $X$:
one consisting of all faces of $X$ with 0 inserted in the $i^{th}$ digit,
and one consisting of all faces of $X$ with 1 inserted in the $i^{th}$ digit.
We denote these subcomplexes by $\PO^i X$ and $\PI^i X$ respectively.

\begin{definition}
A cubical complex $X$ is a \emph{near-prism in direction i}
if (1) the boundary of every face of $\Del{i}{X}$ is contained in
$\Link{i}{X}$, and (2) both $\PO^i(\Del{i}{X})$ and $\PI^i(\Del{i}{X})$ are contained
in $X$.
\end{definition}

If $X$ is a near-prism in direction $i$, then it admits the decomposition
\begin{equation}
\label{eq:decomp}
X = \PO^i(\Del{i}{X}) \cup \PI^i(\Del{i}{X}) \cup \Prism_i(\Link{i}{X}).
\end{equation}

The following fact is not difficult to prove directly from the definition of a shifted cubical complex.

\begin{lemma}
\label{lemma:near}
$X$ is a shifted cubical complex iff $X$ is a near-prism in direction $1$ and both $\Link{1}{X}$ and $\Del{1}{X}$ are shifted cubical complexes with respect to the direction set $\{2,\ldots,n\}$.  
\end{lemma}

In particular, shifted cubical complexes are
iterated near-prisms.  As in~\cite{DR}, this characterization
will help describe the Laplacian eigenvalues of a shifted cubical
complex.

First, we introduce notation to work with weakly decreasing sets
of nonnegative integers, defined up to $\zeq$-equivalence (i.e.,
with an indeterminate number of trailing zeroes).
The symbol $2^m$ will denote the sequence $(2,2,\dots,2)$ of length~$m$.
If ${\bf a} = (a_1\geq a_2\geq\ldots)$
and ${\bf b}= (b_1\geq b_2\geq\ldots)$
are two sequences, then ${\bf a} + {\bf b}$ denotes
the sequence $(a_1 + b_1 \geq a_2 + b_2 \geq \ldots)$.  In
particular, $2^m + {\bf a}$ is the sequence derived from ${\bf a}$
by adding 2 to each of the first $m$ entries of ${\bf a}$, padding
the end of ${\bf a}$ with 0's if necessary.  For instance,
$2^8 + (7,5,5,2,1) \zeq (9,7,7,4,3,2,2,2)$.

Let ${\bf{s}}(X)$ denote the sequence of non-zero
eigenvalues of $\Lud_{X,d-1}$, written in weakly decreasing order.

\begin{theorem}
\label{thm:near-prism-del-link}
Let $X^d$ be a pure cubical complex which is a near-prism in direction $1$, and
let $\ell$ denote the number of facets of $\Link{1}{X}$.
If $X$ is also a prism, then 
$${\bf{s}}(X) =  2^\ell +  {\bf{s}}(\Link{1}{X}).$$
Otherwise,
$${\bf{s}}(X) = {\bf{s}}(\Del{1}{X}) ~\cup~ (2^\ell + ({\bf{s}}(\Del{1}{X}) \cup {\bf{s}}(\Link{1}{X}))).$$
\end{theorem}

\begin{proof}
First, suppose that $X$ is not a prism.  Then $\dim(\Del{1}{X})=d$ and
$\dim(\Link{1}{X})=d-1$. Let $X' = \Del{1}{X}.$
Since $X$ is a pure near-prism, we see that 
  $$X'_d = \Del{1}{X} \, \textrm{ and } \,  X'_{d-1} =  \Link{1}{X}.$$
So $(\Prism_1X')_d = \PO^1\Del{1}{X} \cup \PI^1\Del{1}{X} \cup
\Prism_1 \Link{1}{X} = X$ by equation~\eqref{eq:decomp} above.
Therefore
  \begin{align}
  \sud_{d-1}(X) &= \sud_{d-1}(\Prism X')\notag\\
  &\zeq \sud_{d-1}(X') ~\cup~ (2^{|X'_{d-1}|} + \stot_{d-1}(X'))\notag\\
  &\zeq \sud_{d-1}(X') ~\cup~ (2^{|X'_{d-1}|} + (\sud_{d-1}(X') ~\cup~  \sud_{d-2}(X')))\label{eq:stoppingpoint}\\
  &= \sud_{d-1}(\Del{1}{X}) ~\cup~ (2^{|(\Link{1}{X})_{d-1}|} + (\sud_{d-1}(\Del{1}{X}) \cup \sud_{d-2}(\Link{1}{X}))),\notag
  \end{align}
where the second line uses~\eqref{prism:sud}, and
all integer sequences are listed in weakly decreasing order.

The proof is similar if $X$ is a prism.  In this case, $X=\Prism X'$ where $\dim X' = d-1$, and $X'=\Link{1}{X}$, so again $X'_{d-1}=\Link{1}{X}$.  Equation~\eqref{eq:stoppingpoint} then applies again, but $\sud_{d-1}(X')$ consists of all 0's, because $\dim X' = d-1$.  Therefore
$$
\sud_{d-1}(X) \zeq 2^{|(\Link{1}{X})_{d-1}|} + \sud_{d-2}(\Link{1}{X}).
$$
\end{proof}

\begin{corollary}\label{cor:shifted.integral}
Shifted cubical complexes are Laplacian integral.
\end{corollary}
\begin{proof}
We proceed by induction on the number of directions.  If a shifted
cubical complex~$X$ has only one direction, then $X=Q_1$, which is
Laplacian integral by Theorem~\ref{cube-eigenvalue-theorem}.

Now let $X$ be a $d$-dimensional shifted cubical complex on more than
one direction.
Let $Y$ be the \emph{$j$-dimensional pure skeleton of $X$}, i.e., the subcomplex
consisting of the $j$-dimensional faces of $X$ and all of their subfaces.  It is clear that 
$\sud_{j-1}(X) \zeq \sdu_{j}(X) \zeq \sdu_{j}(Y) \zeq \sud_{j-1}(Y)$, since
$\sdu_{j}$ depends only on $X_j$ and $X_{j-1}$.  Therefore, we only need to show that
$Y$ is Laplacian integral, but $Y$ is pure by definition, so we may apply
Theorem~\ref{thm:near-prism-del-link}.  By induction, $\Del{1}{Y}$ and 
$\Link{1}{Y}$ are Laplacian integral, and we are done.
\end{proof}

In the case of a shifted simplicial complex $\Delta$, it is shown in~\cite{DR}
that $\sud(\Delta) = d^T(\Delta)$, where $d^T$ is the transpose of the
vertex-facet degree sequence, and that $d^T$
satisfies the simplicial analogue of the recursion
of Theorem~\ref{thm:near-prism-del-link}.

\begin{problem}  
 Is there an analogous notion of degree sequence for
cubical complexes that is related to the Laplacian spectrum?
\end{problem}

\subsection{Mirroring}
Let $\Delta$ be a simplicial complex on vertex set $[n]$.  The \emph{mirror} of
$\Delta$ is the cubical complex
  $$M(\Delta) = \{f\in Q_n \st \dir(f)\in\Delta\}.$$
Mirroring was first used by Coxeter~\cite{Coxeter} to study certain generalizations
of regular polytopes;
see~\cite[section 2.1]{BBC} for a nice summary of its history and uses.
It is not hard to see that the mirroring operator takes the class of shifted
simplicial complexes to the class of shifted cubical complexes.  
Mirroring also behaves nicely with respect to other basic operations on
simplicial and cubical complexes:
\begin{enumerate}
\item $M(\textrm{Cone}(\Delta)) = \Prism (M(\Delta))$;
\item $M(\Delta_k) = (M(\Delta))_k$;
\item $M(\Del{i}{\Delta}) = \Del{i}{M(\Delta)}$;
\item $M(\Link{i}{\Delta}) = \Link{i}{M(\Delta)}$;
\item $M(\bnd \Delta) = \bnd M(\Delta)$.
\end{enumerate}
Here $\del$ and $\link$  have their usual meanings for simplicial complexes, i.e.,
$$\Del{i}{\Delta} =\{\sigma\sm i \st \sigma\in\Delta\}, \qquad
\Link{i}{\Delta} =\{\sigma\sm i \st \sigma\in\Delta,\ i\in\sigma\},
$$
and $\bnd X$ denotes the union of the boundary faces of every face of $X$
(equivalently, the set of non-maximal faces of~$X$).

Furthermore, mirroring takes near-cones to near-prisms in the following
sense.
A simplicial near-cone $\Delta$ with apex $i$ has the property that
$\bnd(\Del{i}{\Delta}) \subseteq \Link{i}{\Delta}$.  In this case, 
$$
\bnd(\Del{i}{M(\Delta)}) = \bnd(M(\Del{i}{\Delta})) = M(\bnd(\Del{i}{\Delta}))
\subseteq M(\Link{i}{\Delta}) = \Link{i}{M(\Delta)},
$$
and so we may construct a near-prism (as in \eqref{eq:decomp})
from the cubical complexes
$\Del{i}{M(\Delta)}$ and $\Link{i}{M(\Delta)}$.

Unfortunately, mirroring does not seem to behave nicely with respect to trees
or Laplacian eigenvalues.  For instance, mirroring does not preserve the APC property.
Thus, the mirror of a simplicial complex with spanning trees will not
necessarily have a spanning tree.  Even the mirror of a pure shifted
simplicial complex is not necessarily APC.  
(For example, let $\Delta$ be the graph with vertices $1,2,3$ and
edges $12,13$.  Then $\Delta$ is contractible, but $M(\Delta)$ is
the prism over an empty square --- a 2-dimensional cell complex that is
homotopy equivalent to a circle, hence not APC.)  Mirroring also does not in general 
preserve the property of being Laplacian integral.  (For example,
the pure 2-dimensional complex $\Delta$ with vertices $1,2,3,4,5$
and facets $124,125,134,135$ is a matroid complex, hence
Laplacian integral by \cite{KRS}, but $M(\Delta)$ is not Laplacian
integral.)

Although mirroring does not appear useful for enumerating spanning trees, or computing
eigenvalues, it is possible that there are still interesting things to explore, such
as the relations between eigenvalues of a simplicial complex and its mirror, perhaps
especially in special cases such as shifted or near-cone/near-prism complexes.

\subsection{Homotopy type}
Bj\"orner and Kalai \cite{BK} proved that a simplicial near-cone is always homotopy equivalent
to a wedge of spheres, and that the number of spheres
in each dimension is easily determined from its combinatorial structure. An immediate consequence
is that shifted simplicial complexes are homotopy equivalent to wedges of spheres, with the number of
spheres in each dimension again easy to describe.

Given a near-prism $X$, define 
$$
B_i(X) = \Del{i}{X} \setminus \Link{i}{X}.
$$
By the definition of near-prism, every face in $\PO^i B_i(X)$ or in $\PI^i B_i(X)$ is a facet of $X$.

\begin{conjecture}
\label{near-prism-homotopy-formula}
If $X$ is a near-prism in direction $i$, and $\Del{i}{X}$ is homotopy equivalent to a wedge of spheres,
$$
\Del{i}{X} \simeq \bigvee_j S_j^{r_j},
$$
then $X$ is homotopy equivalent to a wedge of spheres.  Specifically,
$$
X \simeq \bigvee_j S_j^{r_j + c_{i,j}},
$$
where $c_{i,j}$ is the number of $j$-dimensional cells in $B_i(X)$.
\end{conjecture}

\begin{problem}
\label{homotopy_problem}
Assuming the truth of Conjecture~\ref{near-prism-homotopy-formula},
give a combinatorial formula for the homotopy type of a shifted cubical complex.
\end{problem}
Eran Nevo \cite{Nevo} recently showed us a solution to Problem \ref{homotopy_problem}.  In particular, his combinatorial formula confirms Conjecture \ref{near-prism-homotopy-formula} when $X$ is a shifted cubical complex.

\subsection{Extremality}

Shifted simplicial complexes derive their name from the existence of
shifting operators which associate a shifted complex to any simplicial
complex.  Shifted simplicial complexes exhibit certain extremality 
properties with respect to invariants such as $f$-vectors, Betti numbers,
degree sequence and Laplacian eigenvalues: see, e.g., \cite{DR,Kalai_shifting,KR}.  In~\cite{BK},
Bj\"{o}rner and Kalai pose the problem of developing shifting operators
for arbitrary polyhedral complexes.

\begin{problem}
Is there a natural notion of cubical shifting which associates a
shifted cubical complex to an arbitrary cubical complex?  In what
ways are shifted cubical complexes extremal in the class of
all (proper) cubical complexes?
\end{problem}

\section{Duality}
\label{sec:duality}

In this section we examine pairs of dual complexes.  We now extend the definition of cell complex to allow the possibility that the complex contains the empty set as a $(-1)$-dimensional face.  We assume the
reader is familiar with the basics of matroid theory and refer to
\cite{Oxley} for an excellent reference.

Let $X$ and $Y$ be cell complexes.  We say that $X$ and $Y$ are \emph{dual} if there is an inclusion-reversing bijection $f\mapsto f^*$
from the cells of~$X$ to the cells of $Y$ such that, for some $d$, $\dim f+\dim f^*=d$
for every $f$.  (Necessarily, then, $\dim X = d+1$ if $Y$ contains the empty set as a face, and $\dim X = d$ otherwise; similarly, $\dim Y = d+1$ if $X$ contains the empty set as a face, and $\dim Y = d$ otherwise.)  We also require that the boundary maps of $X$ equal the coboundary maps of $Y$:
that is, if we extend the duality bijection to all cellular chains by linearity, then
  $$\bd_X(f)=\cbd_Y(f^*)$$
or, more specifically,
  \begin{align*}
  \bd_{X,i}\colon C_i(X)\to C_{i-1}(X) &\quad=\quad \cbd_{Y,d-i+1}\colon C_{d-i}(Y)\to C_{d-i+1}(Y) \quad\text{and}\\
  \cbd_{X,i}\colon C_i(X)\to C_{i-1}(X) &\quad=\quad \bd_{Y,d-i+1}\colon C_{d-i}(Y)\to C_{d-i+1}(Y).\end{align*}
These equalities imply immediately that
  \begin{equation} \label{duality-formula}
  \Lany_{X,i} = \Lany_{Y,d-i},\qquad
  \Eany_i(X;\q) = \Eany_{d-i}(Y;\q),
  \end{equation}
where $\bullet\in\{\text{ud},\text{du},\text{tot}\}$. 

\begin{proposition}
\label{dual-complexes}
Suppose that $X^d$ and $Y^d$ are dual complexes.  Let
$T\subseteq X_i$ and let $U=\{f^* \st f\in X_i\sm T\}$.
Then the subcomplex $X_T=T\cup X_{i-1}$ is an $i$-spanning tree
of $X$ if and only if the subcomplex $Y_U=U\cup Y_{d-i-1}$ 
is a $(d-i)$-spanning tree of $Y$.
\end{proposition}

\begin{proof}
An equivalent statement is the following: the matroid\footnote{%
When $X$ is a simplicial complex, this matroid is known as a \emph{simplicial
matroid}; see, e.g., \cite{CL}.  (One could analogously define a \emph{cellular matroid}
as a matroid representable by the columns of some cellular boundary map.  This is
a more general class of matroids; for example, simplicial matroids must be simple, while
cellular matroids need not be.)}
 represented (over any field of characteristic~0, say $\Qq$) by the columns of
$M=\bd_{X,i}$ is dual to the matroid represented by the columns of $\bd_{Y,d-i}=\cbd_{X,i+1}$,
or equivalently by the rows of $N=\bd_{X,i+1}$.  Since $\bd_{X,i}\bd_{X,i+1}=0$, the
column span of $\bd_{X,i}$ (regarded as a vector subspace of $\Qq X_i$)
is orthogonal, under the standard inner product, to the row span of $\bd_{X,i+1}$.
On the other hand, these two subspaces have complementary dimension, because
  $$\rank M ~=~ |X_i|-\dim\ker M ~=~ |X_i|-\dim\im M ~=~ |X_i|-\rank N ~=~ |X_i|-\rank N.$$
The desired duality now follows from \cite[Exercise~2.2.10(ii)]{Oxley}.
\end{proof}

Duality carries over naturally to the algebraically weighted setting.
Specifically, let each cell $f\in X$ have an indeterminate weight $\xi_f$.
We define the weighted cellular boundary and coboundary maps
as in Section~\ref{WCMTT-section}, which give rise to
weighted Laplacians and spectrum polynomials as usual.  Then, we assign weights
to the cells of the dual complex $Y$ by the formula
  $$\xi_{f^*}=1/\xi_f.$$
It is routine to check that the formulas of~\eqref{duality-formula} carry over to the weighted setting, that is,
  \begin{equation} \label{duality-weighted}
  \Lwany_{X,i} = \Lwany_{Y,d-i} \qandq
  \Ewany_i(X;\q) = \Ewany_{d-i}(Y;\q),
  \end{equation}
where $\bullet\in\{\text{ud},\text{du},\text{tot}\}$. 

Furthermore, the matroid represented by $\wbd_{X,i}$ is identical to that represented
by $\bd_{X,i}$, since we have just adjoined the indeterminates
$\{\xi_f \st f\in X\}$ to the ground field, and then multiplied the rows and columns of the
matrix by nonzero scalars (which does not change the matroid structure).  Therefore, the proof of Proposition~\ref{dual-complexes}
is still valid if we replace the boundary and coboundary maps with their
algebraically weighted analogues.

\subsection{Spectrum polynomials of matroids}
Let $M$ be a matroid on ground set $V$, and let $X$ be the corresponding
independent set complex, that is, the simplicial complex on~$V$ whose
facets are the bases of~$M$.
Kook, Reiner and Stanton~\cite{KRS} defined the \emph{spectrum polynomial} of~$M$ to be
  $$\Spec_M(t;\q)=\sum_I t^{r(I)}\q^{|\langle I_1\rangle|},$$
where $r$ is the rank function of~$M$; $A\mapsto\langle A\rangle$ is its closure operator;
and $I_1$ is a certain subset of $I$ defined algorithmically (the details
are not needed in the present context).  The algorithm
depends on a choice of a total ordering for $V$, but $\Spec_M(t;\q)$
does not.  By \cite[Corollary~18]{KRS}, the spectrum polynomial records
the Laplacian eigenvalues of~$X$, via the formula
  \begin{equation} \label{eigenvals-from-specP}
  \sum_i\sum_{\lambda\in\Stot_i(X)} t^i \q^\lambda ~=~  \Etot(X;\q, t) ~=~ t^{-1} \q^{|V|} \Spec_M(t;\q^{-1}).
  \end{equation}

We are going to apply this result to calculate the Laplacian spectra
of a complete colorful complex, and thereby recover Adin's
torsion-weighted count of their simplicial spanning
trees~\cite{Adin}.  We will see that the dual to the $n$-cube arises
as a special case of a complete colorful complex.

Let $a_1,\dots,a_n$ be positive integers.  The \emph{complete colorful complex}
$X=X(a_1,\dots,a_n)$ is defined as follows.
Let $V_1=\{v_{1,1},\dots,v_{1,a_1}\},\dots,V_n=\{v_{n,1},\dots,v_{n,a_n}\}$ be~$n$ pairwise
disjoint vertex sets of cardinalities $a_1,\dots,a_n$.  We will regard each $V_i$
as colored with a different color.  Then $X$
is the pure $(n-1)$-dimensional simplicial complex on $V=V_1\cup\cdots\cup V_n$ defined by
  $$X = \{f\subseteq V\st |f\cap V_i|\leq 1\quad\forall i\}.$$
Define the \emph{color set} of an independent set $I$ to be the set
$C(I)=\{i \st I\cap V_i\neq\0\}$.  Also, for a set of colors $K\subseteq[n]$,
abbreviate $\bar K = [n]\sm K$ and $A(K) := \sum_{i\in K} a_i$.

Observe that $X$ is a matroid independence complex of a very simple form; the
matroid~$M$ is just the direct sum of $n$ rank-1 matroids
whose ground sets are the color classes~$V_i$.  In particular,
$X$ is Cohen-Macaulay \cite[pp.~88--89]{CCA}, so $\HH_i(X,\Zz)=0$
for every $i\leq n-2$.

If $n=2$, then $X$ is
a complete bipartite graph, while if $a_i=2$ for every~$i$ then $X$
is the boundary of an $n$-dimensional cross-polytope.

Let $W=\{v_{1,1},v_{2,1},\dots,v_{n,1}\}$.
For any independent set $I\subseteq V$ (i.e.,
any face of $X$), the subset $I_1$ produced
by the algorithm of~\cite{KRS} is just
$I_1 = I\sm W$; in particular,
$|\langle I_1\rangle|=A(C(I_1))$.  Therefore,
  \begin{align}
  \Spec_M(t;\q) &= \sum_{K\subseteq[n]}\quad \sum_{I\in X\st C(I)=K} t^{r(I)}\q^{|\langle I_1\rangle|}\notag\\
    &= \sum_{K\subseteq[n]} t^{|K|}\sum_{I\in X\st C(I)=K} \q^{|I \sm W|}\notag\\
    &= \sum_{K\subseteq[n]} t^{|K|} \prod_{k\in K} (1+(a_k-1)\q^{a_k}). \label{colorful-eigenvalues}
  \end{align}
Applying the Kook-Reiner-Stanton theorem \eqref{eigenvals-from-specP} and extracting
the coefficient of $t^i$ (for $-1\leq i\leq n-1=\dim X$), we obtain
  \begin{align}
  \Etot_i(X;\q)
 &= \sum_{\substack{K\subseteq[n]\\|K|=i+1}} \q^{A(\bar K)} \left(\prod_{k\in K} (\q^{a_k}+a_k-1)\right) \label{CCC-eigenvals}\notag\\
 &= \sum_{\substack{K\subseteq[n]\\|K|=i+1}} \left(\sum_{J\subseteq K}\left(\prod_{j\in K\sm J} (a_j-1)\right) \q^{A(\bar K)+A(K\sm J)}\right)\notag\\
 &= \sum_{\substack{K\subseteq[n]\\|K|=i+1}} \left(\sum_{J\subseteq K}\left(\prod_{j\in K\sm J} (a_j-1)\right) \q^{A(\bar J)}\right).
   \end{align}
 Therefore,
  \begin{equation} \label{colorful-omega}
  \omega_i(X) = \prod_{\substack{K\subseteq[n]\\|K|=i+1}} \left. \prod_{J\subseteq K} A(\bar J)^{\displaystyle\left[\prod_{j\in J}(a_j-1)\right]}\right.
  = \prod_{\substack{J\subseteq[n]\\|J|\leq i+1}} A(\bar J)^{\displaystyle\left[\binom{n-|J|}{i+1-|J|}\prod_{j\in J}(a_j-1)\right]}.
  \end{equation}
Plugging \eqref{colorful-omega} into~\eqref{tau-from-omega} gives
\begin{align}
\tau_k(X)
~&=~ \prod_{i=0}^{k-1} \prod_{\substack{J\subseteq[n]\\|J|\leq i+1}}
  A(\bar J)^{\displaystyle\left[ (-1)^{k-i+1}(k-i)\binom{n-|J|}{i+1-|J|}\prod_{j\in J}(a_j-1)\right]}\notag\\
~&=~ \prod_{j=0}^n \prod_{\substack{J\subseteq[n]\\|J|=j}}
  B(J)^{\displaystyle\left[\sum_{i=j-1}^{k-1} (-1)^{k-i+1}(k-i)\binom{n-j}{i+1-j}\right]}\label{enormous-parentheses}
\end{align}
where
  $$B(J) ~=~  A(\bar J)^{\displaystyle\left(\prod_{j\in J}(a_j-1)\right)}.$$
The bracketed exponent in \eqref{enormous-parentheses} can be simplified,
first by rewriting it in terms of the three quantities
$N=n-j$, $M=k-j$, $a=i+1-j$, and then by some routine calculations
(which we omit) using Pascal's recurrence:
  \begin{align*}
  \sum_{i=j-1}^{k-1}(-1)^{k-i+1}(k-i)\binom{n-j}{i+1-j}
  &= (-1)^M\sum_{a=0}^{M}(-1)^a(M-a+1)\binom{N}{a}\\
  &= (-1)^M\sum_{a=0}^{M}(-1)^a\binom{N-1}{a}\\
  &= \binom{N-2}{M} ~=~ \binom{n-j-2}{k-j}.
  \end{align*}
This is 
precisely the exponent that appears in Adin's formula
for $\tau_k(X)$ \cite[Theorem~1.5]{Adin}.  (Adin's $r$ is
our $n$, and Adin's $d$ is our $j$.)
In particular, it is zero when $j>k$.
Rewriting~\eqref{enormous-parentheses} recovers Adin's
formula for the tree enumerators of complete colorful complexes:

\begin{theorem}[Adin]
\label{count-colorful-trees}
Let $X$ be the complete colorful complex
with color classes of sizes $a_1,\dots,a_n$,
and define $B(J)$ as above.  Then
$$
\tau_k(X) ~=~
\prod_{j=0}^k \prod_{\substack{J\subseteq[n]\\|J|=j}}
  B(J)^{\binom{n-j-2}{k-j}}.
$$
\end{theorem}

In the special case that $a_i=2$ for all~$i$ (so $X$ is the boundary
of an $n$-dimensional cross-polytope), the formula~\eqref{colorful-eigenvalues} specializes to
  $$
  \Spec_M(t;\q) =  \sum_{k=0}^n \binom{n}{k} t^k (1+\q^2)^k
  $$
and then applying \eqref{eigenvals-from-specP}, we see that
  $$
  \sum_i\sum_{\lambda\in\Stot_i(X_n)} t^i \q^\lambda ~=~  \sum_{k=0}^n t^{k-1} \binom{n}{k} \q^{2n-2k}(1+\q^2)^k
 ~=~  \sum_{j=-1}^{n-1} t^j \binom{n}{j+1} \q^{2n-2j-2}(1+\q^2)^{j+1}
  $$
or equivalently
  $$\Etot_j(X_n;\q) = \binom{n}{j+1} \q^{2n-2j-2}(1+\q^2)^{j+1}.$$
The dual complex to an $n$-dimensional cross-polytope is the $n$-cube $Q_n$.
By formula~\eqref{duality-formula} (with $d=n-1$), we have therefore
  $$\Etot_k(Q_n;\q) ~=~ \Etot_{n-1-k}(X_n;\q) ~=~ \binom{n}{k} \q^{2k}(1+\q^2)^{n-k},$$
giving another proof of (the first formula of) Theorem~\ref{cube-eigenvalue-theorem}.


\end{document}